\documentclass[preprint]{imsart}

\RequirePackage[OT1]{fontenc}
\RequirePackage{amsthm,amsmath}
\RequirePackage[authoryear]{natbib}
\usepackage[utf8]{inputenc} 
\usepackage[T1]{fontenc}    
\usepackage{hyperref}
\usepackage{url}            
\usepackage{booktabs}       
\usepackage{amsfonts}       
\usepackage{nicefrac}       
\usepackage{microtype}      
\usepackage{lipsum}
\usepackage[english]{babel}
\newcommand{\wwidehat}[1]{\widehat{\widehat{#1}}}

\usepackage{amssymb}
\usepackage{amsmath}
\usepackage{comment}
\usepackage{amsthm}
\usepackage{bbm}
\usepackage{mathabx}
\usepackage{verbatim}

\newtheorem{theorem}{Theorem}[section]
\newtheorem{remark}{Remark}
\newtheorem{assumption}{Assumption}
\newtheorem{corollary}{Corollary}[theorem]
\newtheorem{lemma}[theorem]{Lemma}
\theoremstyle{definition}
\newtheorem{definition}{Definition}[section]
\newcommand{\norm}[1]{\left\lVert#1\right\rVert}
\newcommand{\tildehat}[1]{\tilde{\widehat{#1}}}

\newenvironment{itquote}
{\begin{quote}\itshape}
	{\end{quote}\ignorespacesafterend}

\arxiv{arXiv:0000.0000}

\startlocaldefs
\numberwithin{equation}{section}
\theoremstyle{plain}

\endlocaldefs

\begin{document}

\begin{frontmatter}
\title{Model-free Bootstrap for a General Class of  Stationary Time Series}
\runtitle{Model-free Bootstrap for    Stationary   Series}

\begin{aug}
\author{\fnms{Yiren} \snm{Wang}\ead[label=e1]{yiw518@ucsd.edu}},
\author{\fnms{Dimitris} \snm{N. Politis}\ead[label=e2]{dpolitis@ucsd.edu}}
\affiliation{University of California, San Diego}
\runauthor{Y. Wang and D.N. Politis}
\address{Department of Mathematics\\
	University of California, San Diego\\
	La Jolla, CA 92093-0112, USA\\
	email: {\normalfont \color{blue} yiw518@ucsd.edu; dpolitis@ucsd.edu}
	}

\end{aug}

\begin{abstract}
	A model-free bootstrap procedure   for a general class of stationary time series
is introduced.  The theoretical framework is established, showing asymptotic validity of bootstrap confidence intervals for many statistics of interest. In addition,   asymptotic validity of  one-step ahead 
bootstrap prediction intervals is also demonstrated. Finite-sample experiments   are  
conducted to empirically confirm the performance of the new method, and to compare with 
 popular methods such as the block bootstrap and the autoregressive (AR)-sieve bootstrap.
\end{abstract}
\end{frontmatter}

\section{Introduction}
 The bootstrap, since its introduction by \cite{efron1979}, has been an invaluable tool for statistical inference
with independent data. 
Resampling for time series has also been a flourishing topic since the  late 1980s.
However, there is a plethora of ways to resample a stationary time series.
It is always important to validate the correctness of such bootstrap procedures, i.e. show their asymptotic validity and range of applicability with respect to common statistics. These problems have been well studied for popular methods like the block bootstrap and the autoregressive (AR)-sieve bootstrap.
For a summary of the state-of-the-art,  see \cite{McElroyPolitis2019}
and \cite{KreissPaparoditis2020}. 

 In a dependent setup, the main purpose of bootstrap is two-fold: one is to obtain confidence intervals for a parameter of interest and/or conduct a hypothesis test.
Another important aspect of time series analysis  is forecasting. A standard setup is the following: given the time series data $\{Y_t\}_{t =1 }^n$,  the goal is $h$-step ahead  prediction, i.e., predicting  
$Y_{t+h}$ for some integer $h\geq 1$. 

An optimal 
 $h$-step ahead point predictor $\widehat{Y}_{t+h}$ should minimize the expected loss between the true   $Y_{t+h}$ and itself, conditioned on the current data $\{Y_1,\cdots,Y_n\}$. The most widely used loss functions are $L^1$ and $L^2$. The $L^2$ loss,  $E\left(\left(\widehat{Y}_{t+h} - Y_{t+h}\right)^2 | Y_1,\cdots, Y_n\right)$ is minimized by $\widehat{Y}_{t+h} = E(Y_{n+h}|Y_1,\cdots Y_n)$. The  $L^1$ loss,  $E\left(| \widehat{Y}_{t+h} - Y_{t+h}|  | Y_1,\cdots, Y_n\right)$ is minimized by
the conditional median $med(Y_{t+h}|Y_1,\cdots,Y_n)$  instead.

 Besides point predictors, prediction intervals and joint prediction regions are quite useful; since any point predictor will invariably incur an error, it is important to provide a range of values where the 
future point $Y_{t+h}$ will be found with high probability.   Prediction intervals can be
  constructed by approximating the distribution of the so-called {\it predictive root}, i.e. $Y_{t+h} - \widehat{Y}_{t+h}$, and using the respective quantiles to produce upper and lower bounds. 
Approximating this  distribution typically requires one to fit a specific model to the data, which enables
a model-based resampling for $Y_{t+h}$ and $\widehat{Y}_{t+h}$ separately; see 
\cite{PanPolitis2016} for a review. 

However, model-fitting and prediction are two separate notions with very different objective functions.
Cross-validation ideas that are currently popular attempt to link the two notions, in choosing a model that 
is actually good for predictive purposes. 
Nevertheless, it is possible for the practitioner to proceed directly to prediction without the
intermediary step of model-fitting; this is the essence of the {\it model-free} prediction principle
of  \cite{Politis2013}, \cite{Politis2015}. To describe it, the goal is to find an invertible transformation that  transforms 
the data vector  $(Y_1, \ldots, Y_n)^\prime$  to a new data vector  $(e_1, \ldots, e_n)^\prime$ 
whose entries are independent, identically distributed (i.i.d.).
One can then employ the i.i.d.~bootstrap on 
 the $e_1, \ldots, e_m$ to generate  $e_1^*, \ldots, e_m^*$, 
and  use the inverse transform to get bootstrap samples $Y_1^*, \ldots, Y_m^*$  in
the domain of  the original data. 
Using $m=n$ is the standard framework for estimation and confidence intervals;
interestingly, using $m=n+h$ allows us to equally address the problem of
forecasting  $Y_{t+h}$ with prediction intervals. 

 Under regularity conditions, such a transformation  always exists but is not unique; 
see Ch. 2.3.3 of \cite{Politis2015}.
The challenge for the practitioner is to use the   structure of the data at hand in order to devise
a   transformation that works in the given setting, having features that can be estimated from the data.
In a model-based approach, these steps are analogous to choosing a model, and then fitting the model
using the data. Indeed, any model driven by i.i.d.~errors can be used to define a transformation of the data 
towards the i.i.d.~target; however,  the power of the model-free approach is that it can work
without restricting oneself to a model equation.  

To elaborate,  if the data arise as a stretch of a strictly stationary time series $\{Y_t\}$ 
with (absolutely) continuous distributions, then the Rosenblatt transformation (\cite{rosenblatt1952}) can
be used to transform  $ Y_1, \ldots, Y_n $ to a set of $n$ i.i.d.~Uniform random variables. 
In general, this application of the  Rosenblatt transformation can not be implemented in practice because
it involves $n$ unknown conditional distribution functions. However, if additional structure is assumed, 
e.g., when $\{Y_t\}$  stationary Markov sequence,
then this approach is feasible; see \cite{PanPolitis2016b} and Ch. 8 of \cite{Politis2015}.

To describe a different approach, recall
the Linear Process Bootstrap (LPB) of \cite{mcmurrypolitis2010} which essentially transforms the 
the  data vector  $(Y_1, \ldots, Y_n)^\prime$  to a   data vector  $(W_1, \ldots, W_n)^\prime$
that has uncorrelated entries, i.e., $\{W_t\}$ is a  `white noise'.
If $\{Y_t\}$  is a linear time series, then $\{W_t\}$ can further be claimed to be i.i.d.~(under some
conditions).   The LPB has parallels with the 
AR-sieve bootstrap since both are applicable to nonlinear time series as long as the
statistic of interest has a large-sample distribution   that only depends on the
first and second order moment structure of the data; see \cite{kreiss2011}
and \cite{Jentsch}.

Nevertheless,  in the search of
a transformation that renders the data i.i.d., it may be helpful to first devise a
transformation into Gaussianity; see e.g., Ch. 2.3.2 of \cite{Politis2015}.
For example, we can use a version of the Probability Integral Transform (PIT) in order 
to transform our time series data to Gaussian; the latter can then be transformed to i.i.d.~by
a decorrelating/whitening operation as in the LPB. This approach was first suggested in 
Ch. 9 of \cite{Politis2015}, and was practically implemented to the setting of
a locally stationary  time series by \cite{Das2017}.

In the paper at hand, we focus on  stationary  time series data, with the goal of 
establishing the realm of applicability  of the above mentioned procedure
which, for lack of a better word, we will call the {\it model-free bootstrap} (MFB). 
We will show   asymptotic validity of the MFB for a general class of stationary processes, and
for many types of statistics of interest. We will also establish MFB's validity for the
construction of one-step-ahead prediction intervals, i.e., to fix ideas we will focus on the
case $h=1$ in the above.

The remaining of the paper is organized as follows. Section 2 restates the MFB algorithm 
carefully. Section 3 introduces some necessary tools and assumptions to be used,    and summarizes some useful preliminary results for our proofs. Section 4 proves MFB's asymptotic validity 
for various estimation problems, while Section 5 shows its  validity for prediction
intervals. Numerical experiments that back up our asymptotic results are presented in Section 6.
Technical proofs are given in the Appendix.

\section{Model-free bootstrap algorithm}
\subsection{The MFB algorithm}
Here we describe the model-free bootstrap (MFB) algorithm for inference and prediction as  
proposed in Chaper 9 of \cite{Politis2015}. 
Given a time series $\{Y_t\}_{t \in \mathbb{Z}}$ that is strictly stationary, let $F_Y$ be the cumulative distribution function (CDF) of $Y_0$. The PIT defined by
$$U_t = F_Y(Y_t)$$
implies that $U_t$
is uniformly distributed on $[0,1]$, assuming $F_Y$ is continuous. See \cite{PITresult}. Let $\Phi$ be the CDF of standard normal distribution, and $\Phi^{-1}(p) = \inf\{x\in\mathbb{R}: \Phi(x) \geq p\}$ is the quantile function;
  then, $Z_t = \Phi^{-1}(U_t)$ is $\mathcal{N}(0,1)$ distributed. Also, stationarity is preserved for $\{U_t\}$ and $\{Z_t\}$. 

Let $\Sigma_n$ denote the covariance matrix of $\underline{Z}_n = (Z_1,\cdots,Z_n)$, and denote by 
$\Sigma_n^{-\frac{1}{2}}$   the lower triangular matrix from the Cholesky decomposition of $\Sigma_n^{-1}$.
Then, $\underline{\xi}_n = \Sigma_n^{-\frac{1}{2}}\underline{Z}_n$  is a vector of
i.i.d. $\mathcal{N}(0,1)$ entries, provided $Z_1,\cdots,Z_n$ are {\it jointly} normal. 

   Suppose we use a resampling scheme 
to  create  the i.i.d.~bootstrap sample  $ \xi_1^{*},\xi_2^{*},\cdots,\xi_n^{*} $.
Then,  letting $\underline{Z}_n^* = \Sigma_n^{\frac{1}{2}}\underline{\xi}_n^*$
where $\underline{\xi}_n^*=(\xi_1^{*},\xi_2^{*},\cdots,\xi_n^{*})^\prime $,
 and $Y_t^* = F_Y^{-1}\left(\Phi(Z_t^*)\right)$,  then $\{Y_t^*\}$ is our bootstrapped sample.  

Moreover, $\xi_{n+1}^*$ can also be generated through i.i.d. sampling, and $Z_{n+1}^*$ can be generated through the relation $(\underline Z_n, Z_{n+1}^*) = \Sigma_{n+1}^{1/2}(\underline\xi_n,\xi_{n+1}).$ Using the inverse of the previously mentioned transforms, the next bootstrap value can be generated by $Y_{n+1}^* = F_Y^{-1}(\Phi(Z_{n+1}^*))$.  It can be shown that by using   these theoretical transforms, $Y_{n+1}^*|\underline{Y}_n$ has the same distribution as $Y_{n+1}|\underline{Y}_n$.
 
Nevertheless, to use the above steps for practical purposes,  each transform must also be estimated 
in a consistent manner from the data at hand. Furthermore, the validity of the bootstrap procedure has to be investigated, both for estimation and prediction. Thus, several questions arise:
\begin{itemize}
	\item Under what circumstance are the entries of $\underline{Z}_n$ jointly normal?
	\item What estimators for $F_Y$ and $\Sigma_n$ should we use so that the above steps lead to validity of the  bootstrap?
	\item How should we create the i.i.d. bootstrap values $\{\xi_t^*\}$?
\end{itemize}
The first two points will be  addressed in the following paragraphs. For the third point, \cite{Politis2015} has  proposed two ways to do it. One way is sampling with replacement from $\{\widehat{\xi}_t\}_{t=1}^n$, with $\widehat{\xi}_t$ calculated from $Y_t$  using estimated transform functions. A
second way is to 
generate $\xi_t^*$ as i.i.d. $\mathcal{N}(0,1)$, which is presumably the limiting distribution of $\widehat{\xi}_t$. The first method is called \textbf{model-free} (MF), and the second is referred to as \textbf{limit model-free} (LMF) since the limit distribution is used. 

Frequently used notations include the following.
 Let $\widehat{F}$  and $\widehat{\Sigma}_n$ denote general estimators for $F_Y$ and $\Sigma_n$ respectively. The subscript $Y$ is dropped from $\widehat F$ for simplicity. $\Phi$ is the CDF of a standard normal distribution with $\Phi^{-1}$ its quantile function. Let $\tilde{\Phi}$ be the CDF of a thresholded standard normal distribution: suppose $X \sim \mathcal{N}(0,1)$, $X_c = X$ for $|X| \leq c$ and $X_c = sgn(X)c$ for $|X|>c$, where $sgn(\cdot)$ is the sign function. Then $\tilde\Phi$ denotes the CDF of $X_c$ and its inverse $\tilde \Phi^{-1}$ the quantile function. We omit $c$ in the notation for simplification.  Asymptotically we also require $c\rightarrow\infty$ such that $\tilde\Phi^{-1}$ converges to $\Phi^{-1}$. The reason of this augmentation is provided in Section 2.4 and asymptotic details are explained in Section 3. 
 
 By using these practical transforms, we can calculate $\widehat U_t = \widehat{F}(Y_t)$,
and  $\tilde{\widehat{Z}}_t = \tilde{\Phi}^{-1}(\widehat U_t)$, which are the estimations for the latent series $\{U_t\}$ and $\{Z_t\}$ respectively. Since $\{Z_t\}$ is 
latent,  $\widehat{\Sigma}_n$ can not be directly calculated. Instead, we use  $\wwidehat{\Sigma}_n$ which is the same estimator calculated based on $\{\tilde{\widehat{Z}}_t\}$.

 Let $\sigma_Z(k) = E Z_0Z_k$ be the lag-{\it k} autocovariance 
 of $Z_t$, $\widehat{\sigma}_Z(k)$ be its estimator 
 and $\wwidehat{\sigma}_Z(k)$ be the estimator calculated from $\tildehat{Z}_t$. Let $\norm{\cdot}_p = E(|\cdot|^p)^{\frac{1}{p}}$ denote the $p-norm$ of a random variable; $\norm{\cdot}_{op}$ denotes the operator norm of a matrix,
i.e., $\norm{M}_{op} = \sup_{x\in\mathbb{R}^n, \norm{x}_2 = 1} \norm{Mx}_2$
where  $M$ is a $n\times n$ square matrix.   
Relative quantities in the bootstrap world will be denoted by a superscript $*$. 
 
Given the above introduction, we can now describe the model-free bootstrap algorithm. 

\subsection{MFB for confidence intervals} Let $\theta_0$ be a population parameter of interest; $\widehat\theta_n$ an estimator of $\theta_0$ from data $\{Y_t\}_{t=1}^n$, and $\widehat\theta_n^*$ the  same estimator from bootstrapped data $\{Y_t^*\}_{t=1}^n$. 
Define the real-world {\it root}  $r = \theta_0 -  \widehat\theta_n$;
let $R$ denote its CDF. Then,  a $(1-\alpha)100\%$
equal-tailed confidence interval (CI) for $\theta_0$ is 
$$(\widehat \theta_n +   R^{-1}(\alpha/2), \widehat\theta_n +   R^{-1}(1-\alpha/2)),$$
 where  $  R^{-1}(x) = \inf\{r\in\mathbb{R}:  R(r) \geq x\}$ denotes the  quantile function of $  R$. The distribution $R$ could be approximated through bootstrap simulations.
\\

{\small\textbf{Algorithm 1} (Model-free   bootstrap for parameter inference)}
\begin{itquote}
\begin{enumerate}
	\item Given data $\{Y_t\}_{t=1}^n$,  let $\widehat U_t = \widehat{F}(Y_t)$;  $\tilde{\widehat{Z}}_t = \tilde{\Phi}^{-1}(\widehat U_t)$;
	$\widehat{\underline\xi}_n = \wwidehat{\Sigma}_{n}^{-\frac{1}{2}}\tilde{\widehat{\underline Z}}_n$.
	\item (MF) Let $\xi^*_{t}$ be i.i.d. samples from $\bar F_{\widehat\xi}$, where $\bar F_{\widehat\xi}$ is the empirical CDF of $\{\widehat\xi_t\}_{t=1}^n$, and $\underline Z_n^* = \wwidehat{\Sigma}_{n}^{\frac{1}{2}} \underline \xi^*_n$. Let $Y_t^* = \widehat{F}^{-1}(\Phi(Z_t^*))$. Calculate the bootstrap root $r^* =  \widehat{\theta}_n - \widehat\theta^*_n $.

	\item Do the above step $B$ times to form an empirical CDF $\Bar{R}$ based on the
$B$ replicates of $r^*$.  $\Bar{R}$ is used to approximate $R$; hence, 
	an approximate $(1-\alpha)100\%$ CI for $\theta_0$ is 
$$(\widehat\theta_n + \Bar R^{-1}(\alpha/2), \widehat\theta_n + \Bar R^{-1}(1-\alpha/2)).$$
\end{enumerate}
\end{itquote}
For limit model-free (LMF) bootstrap, replace   step 2  with following:\\
\begin{itquote}
2.\:\:(LMF) Let $\xi^*_{t}$ be i.i.d. samples from $\mathcal{N}(0,1)$, and $\underline Z_n^* = \wwidehat{\Sigma}_{n}^{\frac{1}{2}} \underline \xi^*_n$. Let $Y_t^* = \widehat{F}^{-1}(\Phi(Z_t^*))$. Calculate the bootstrap root $r^* =  \widehat{\theta}_n - \widehat\theta^*_n $.
\end{itquote}
As an alternative,   it may be easier to do:
\begin{itquote}
Let $\underline Z_n^* = \mathcal{N}(\mathbf{0},\wwidehat{\Sigma}_{n})$.
\end{itquote}
in order to omit finding the Cholesky decomposition of the covariance matrix $\wwidehat\Sigma_n$. (However in practice, finding the Cholesky decomposition of a covariance matrix is required for generating multivariate normal samples.)
\\

Below are some simple examples of statistics of interest.
\begin{itemize}
		\item The mean: $\theta_0 = E(Y_0)$; $\widehat{\theta}_n =\frac{\sum_{t=1}^n Y_t}{n}$. 
		\item Autocovariance: $\theta_0 = \gamma_Y(k)$; $\widehat{\theta}_n = \widehat{\gamma}(k) =  \frac{1}{n}\sum_{t=1}^{n-k} (Y_t - \bar Y_n)(Y_{t+k} - \bar Y_n)$.	
		\item Autocorrelation: $\theta_0 = \rho(k) = \frac{\gamma(k)}{\gamma(0)}$;  $\widehat{\rho}(k) = \frac{\widehat\gamma(k)}{\widehat\gamma(0)}$.
\end{itemize}
Additional examples   will be addressed in Section 4. 

\subsection{MFB for prediction intervals} For prediction problems, we want to use the  bootstrap to simulate the distribution of $Y_{n+1}$ \textbf{conditional} on past values $\{Y_t\}_{t=1}^n$.  For this purpose, we use bootstrap to approximate the conditional distribution of the predictive root $Y_{n+1} - \widehat Y_{n+1}$, where $\widehat Y_{n+1}$ is a predictor chosen by the practitioner. Let $G$ denote the conditional distribution of the predictive root defined above. Then a $(1-\alpha)100\%$ equal tailed prediction interval for $Y_{n+1}$ is :
$$(\widehat Y_{n+1} + G^{-1}(\alpha/2), \widehat Y_{n+1} + G^{-1}(1-\alpha/2))$$
Model-free bootstrap algorithm for one step ahead prediction is the following:\\

{\small \textbf{Algorithm 2} (Model-free bootstrap for 1-step ahead prediction)}
\begin{itquote}
	\begin{enumerate}
		\item Given data $\{Y_t\}_{t=1}^n$,  let  $\widehat{U}_t = \widehat F(Y_t)$, $\tildehat Z_t = \Tilde{\Phi}^{-1}(\widehat U_t)$.
		\item 
		Denote $\wwidehat\Sigma_{11} = \wwidehat\Sigma_{n}$, 
		$\wwidehat\Sigma_{12} = \begin{bmatrix}
		\wwidehat\sigma(n)\\ \vdots\\ \wwidehat\sigma(1)
		\end{bmatrix}$,
		 and $\wwidehat\Sigma_{22} = \wwidehat\sigma(0)$.  
		 Let $\wwidehat{\Sigma}_{n+1} = \begin{bmatrix}
		\wwidehat{\Sigma}_{11} & \wwidehat{\Sigma}_{12}\\
		\wwidehat{\Sigma}_{12}^T & \wwidehat{\Sigma}_{22}
		\end{bmatrix}$.
		Let $\{\xi^*_{t}\}_{t=1}^{n+1}$ be drawn randomly with replacement from $\{\widehat\xi_t\}_{t=1}^n$,
		 and $\underline Z_{n}^* = \wwidehat{\Sigma}_{n}^{\frac{1}{2}} \underline \xi^*_n$.
		Let $Z_{n+1}^*$ be the $(n+1)^{th}$ element of the vector 
		$\wwidehat{\Sigma}_{n+1}^{\frac{1}{2}}(\widehat\xi_n,\xi_{n+1}^*)$. 
		Denote the distribution of $Z_{n+1}^*$ as $\widehat{F}_Z^{(n+1)}$. This is also the estimated conditional distribution of $Z_{n+1}|\underline Y_n$.
		The form of this distribution is conditional on our data $\underline Y_n = (Y_1,\cdots,Y_n)$.
 Let $Y_{n+1}^* = \widehat{F}^{-1}(\Phi(Z_{n+1}^*))$. 
		\item Choose a predictor $\widehat Y_{n+1}$ for $Y_{n+1}$ based on $\underline Y_n $.
 For example, the $L^2$ optimal predictor as mentioned in Section 1 is the  expectation of $Z_{n+1}$ conditioning on $\underline Y_n$
that can be approximated by		
		$$\widehat{Y}_{n+1} = \int \widehat{F}^{-1}(\Phi(z)) d\widehat{F}_Z^{(n+1)}(z).$$
		 The above integral can be evaluated through Monte-Carlo simulation. The chosen predictor will be used as the center of our prediction interval, and the bootstrap procedure will be used to capture the distribution of the predictive root in the next steps.
		
		\item Re-estimate all the transforms, matrices and the distribution $\widehat{F}_Z^{(n+1)}$ used in the above calculation, with bootstrapped data $\underline Z_n^{*} = (Z_1^*,\cdots,Z_n^*)$ and $Y_t^* = \widehat{F}^{-1}(\Phi(Z_{t}^*))$. Let $(\widehat{F}_Z^{(n+1)})^*$ denote the re-estimated distribution function for $Z_{n+1}^*$ with bootstrap data $\underline Y_n^* = (Y_1^*,\cdots,Y_n^*)$.
		 Let $(Z_{n+1}|Y_n)^*$ denote  the random variable with estimated conditional distribution $(\widehat{F}_Z^{(n+1)})^*$.
		 
		 Let $\widehat Y_{n+1}^*$ denote the one step ahead predictor with re-estimated transforms based on the bootstrap pseudo data. In the $L^2$-optimal setting, 	$\widehat{Y}_{n+1}^* = \int (\widehat{F}^*)^{-1}(\Phi(z)) d(\widehat{F}_Z^{(n+1)})^*(z)$.
		\item The bootstrapped $L^2-$ optimal predictive root is:
		$$Y_{n+1}^* -\widehat Y_{n+1}^* $$

		\item Denote the empirical CDF of bootstrapped predictive roots as $\bar{G}$. The
approximate  $(1-\alpha)$ prediction interval for $Y_{n+1}$ is
		$$\left(\widehat Y_{n+1} + \bar G^{-1}(\alpha/2),\widehat Y_{n+1} +  \bar G^{-1}(1 - \alpha/2)\right). $$
		
	\end{enumerate}
\end{itquote}

{\small \textbf{Algorithm 3} (Limit model-free bootstrap for 1 step ahead prediction)}
\begin{itquote}
\begin{enumerate}
	\item Given data $\{Y_t\}_{t=1}^n$,  let $\widehat{U}_t = \widehat F(Y_t)$, $\tildehat Z_t = \Tilde{\Phi}^{-1}(\widehat U_t)$.
	\item (LMF)  
	Denote $\wwidehat\Sigma_{11} = \wwidehat\Sigma_{n}$, 	$\wwidehat\Sigma_{12} = \begin{bmatrix}
	\wwidehat\sigma(n)\\ \vdots\\ \wwidehat\sigma(1)
	\end{bmatrix}$. Let $\underline Z_n^{*} \sim \mathcal{N}(0,\wwidehat\Sigma_n)$; $Z_{n+1}^* \sim \mathcal{N}(\wwidehat\Sigma_{21}\wwidehat\Sigma_{11}^{-1}\tilde{\widehat{\underline Z}}_n,\wwidehat\Sigma_{22} - \wwidehat\Sigma_{21}\wwidehat\Sigma_{11}^{-1}\wwidehat\Sigma_{12} )$;
$Y_{n+1}^* = \widehat{F}^{-1}(\Phi(Z_{n+1}^*))$
	\item Choose a predictor for $Y_{n+1}$ based on $\underline Y_n$. For example, the $L^2$ optimal predictor is 
	$$\widehat Y_{n+1}= E (\widehat{F}^{-1}(\Phi(Z_{n+1}))|\underline Y_n).$$
	Where $Z_{n+1}|\underline Y_n \sim \mathcal{N}(\wwidehat\Sigma_{21}\wwidehat\Sigma_{11}^{-1}\tilde{\widehat{\underline Z}}_n,\wwidehat\Sigma_{22} - \wwidehat\Sigma_{21}\wwidehat\Sigma_{11}^{-1}\wwidehat\Sigma_{12} ). $
	\item Re-estimate all the transforms and matrices used in the above calculation, with $\underline Z_n^{*}$ and $Y_t^* = \widehat{F}^{-1}(\Phi(Z_{t}^*))$. The one step ahead predictor in the bootstrap world is $\widehat Y_{n+1}^* = E^*((\widehat{F}^*)^{-1}(\Phi(Z_{n+1}))|\underline Y_n)$,  
	where the expectation in the bootstrap world is calculated through the distribution of $(Z_{n+1}|\underline Y_n)^*$ that is $ \mathcal{N}(\wwidehat\Sigma_{21}^*(\wwidehat\Sigma_{11}^{*})^{-1}\tilde{\widehat{\underline Z}}_n,\wwidehat\Sigma_{22}^* - \wwidehat\Sigma_{21}^*(\wwidehat\Sigma_{11}^{*})^{-1}\wwidehat\Sigma_{12}^* )$ distributed. 
	\item The bootstrapped $L^2-$ optimal predictive root is:
	$$Y_{n+1}^* - \widehat Y_{n+1}^*$$
	
	\item Denote the empirical CDF of bootstraped predictive root as $\bar{G}$. The
approximate  $(1-\alpha)$ prediction interval for $Y_{n+1}$ is
		$$\left(\widehat Y_{n+1} + \bar G^{-1}(\alpha/2),\widehat Y_{n+1} +  \bar G^{-1}(1 - \alpha/2)\right). $$

\end{enumerate}
\end{itquote}
Here we provide an explanation to step \textit{5} above.  Bootstrap is supposed to capture the distribution of the predictive root $Y_{n+1} -  E (\widehat{F}^{-1}(\Phi(Z_{n+1}))|\underline Y_n),$ where for $Y_{n+1} = F^{-1}_Y(\Phi(Z_{n+1}))$, $Z_{n+1}$ has the conditional distribution $\mathcal{N}(\Sigma_{21}\Sigma_{11}^{-1}{\underline Z}_n,\Sigma_{22} - \Sigma_{21}\Sigma_{11}^{-1}\Sigma_{12} ) $; and in the expectation $Z_{n+1}|\underline Y_n \sim \mathcal{N}(\wwidehat\Sigma_{21}\wwidehat\Sigma_{11}^{-1}\tilde{\widehat{\underline Z}}_n,\wwidehat\Sigma_{22} - \wwidehat\Sigma_{21}\wwidehat\Sigma_{11}^{-1}\wwidehat\Sigma_{12} ) $ is estimated from data. Clearly, the randomness in the distribution of the predictive root not only comes from the randomness of the series, but also randomness in the estimation. Thus in the bootstrap world, we should replace theoretical transforms with their data-dependent analogue, and also account for all the errors arising from estimation, i.e. replace $F_Y$ with $\widehat F$, $\widehat{F}$ with $\widehat{F}^*$, and $\wwidehat{\Sigma}_n$ with $\wwidehat{\Sigma}_n^*$, resulting in the formula in step \textit{5}. In this way, the bootstrap procedure can capture all the randomness in the distribution of predictive root, which helps relieve potential {\it undercoverage} issue for finite data.
See also Ch. 9 of \cite{Politis2015}.  

\subsection{Appropriate estimators for $\widehat F$, $\tilde\Phi^{-1}$ and $\wwidehat \Sigma_n$}
 Now we discuss what should be the appropriate estimators $\widehat F$, $\wwidehat \Sigma_n$ and also why an augmented version of  $\Phi^{-1}$ might be needed. Firstly, it is necessary that $\widehat F$, $\widehat F^{-1}$, $\wwidehat{\Sigma}_n$ should be consistent in certain forms for $F_Y$, $F_Y^{-1}$ and $\Sigma_n$ respectively. For $\widehat F$, the first idea is to use the empirical CDF $\bar F(y) = \frac{1}{n} \sum_{t=1}^n I\{Y_t\leq y\}$, where $I\{\cdot\}$ is the indicator function, 
 and its inverse $\bar F^{-1} (p) = \inf\{y\in\mathbb{R}: \bar F(y) \geq p\}$. Under moment and short-range dependence assumptions, consistency of $\bar F$ and $\bar F^{-1}$ can be established by looking into the empirical process and quantile process.  Details are  in later sections and will play an important role  in our   proofs. 

Another natural candidate is the kernel smoothed CDF estimator 
$\widehat{F}_h(y) = \frac{1}{n}\sum_{t=1}^n K_h(y - Y_t)$, where $K_h(y- Y_t) = K(\frac{y - Y_t}{h})$ and $K$ is a smooth CDF function with additional assumptions. The obvious advantage of $\widehat{F}_h$ is that it is continuous, which is a   property  $\bar F$ is lacking.  An additional implication of using $\bar F$ is the resulting $\widehat{U}_t = \bar{F}(Y_t)$ only takes value in $\{\frac{1}{n},\frac{2}{n},\cdots,1\}$ and $Y_t^* = \bar{F}^{-1}(U_t^*)$ only takes value in $\{Y_t\}_{t=1}^n$. But by using the kernel estimator $\widehat F_h$ and its inverse, $Y_t^*$ can take values that did not appear in the original series. 
 If the data size $n$ is large, the influence of this is minimal; whereas if $n$ is small, $\widehat{F}_h$ is a better estimator because of its ability to interpolate unseen values compared to the coarse behavior of $\bar F$. It is also worth mentioning that when sample size is large, using $\bar F$ and its inverse will save computational time comparing to $\widehat F_h$.

Using  $\bar{F}$ for the first step transform $\widehat U_t = \bar F(Y_t)$ will result in having a value $1\in\{\widehat U_t\}_{t=1}^n$, and the following second step transform $\Phi^{-1}(1)$ is not well defined. To this respect, we use an augmented $\tilde\Phi^{-1}$, which is the inverse CDF of a thresholded standard normal $ \mathcal{N}_c(0,1)$ as defined in Section 2.1. 
 By doing this, $\tilde\Phi^{-1}$ is bounded on $[0,1]$, which relieves this problem.
Asymptotically, thresholding also controls the fast diverging behavior of $\Phi^{-1}$ at endpoints $0$ and $1$, which helps in analyzing convergence of the covariance estimator $\wwidehat{\Sigma}_n$ for both $\bar F$ and $\widehat F_h$ scenarios.  
Nevertheless, under finite sample settings, the correction from $\Phi^{-1}$ to $\tilde \Phi^{-1}$ doesn't require too much attention. For $\bar F$ scenario, we can simply change the value $1$ to $\frac{n-1}{n}$ from $\{\widehat U_t\}_{t=1}^n$ to avoid the problem of $\Phi^{-1}(1)$ and the effect of doing this is negligible. And for $\widehat F_h$, no correction is needed.
 
Consistent estimator $\widehat\Sigma_n$ of the autocovariance matrix $\Sigma_n$ has been well studied. \cite{Wu2009} established the first result on consistency of a banded  matrix estimator. 
Here we shall use the more general flat-top estimators of \cite{mcmurrypolitis2010}. Let $\kappa(x)$ be the tapering weight function:
\begin{equation}\label{eq:21}
\kappa(x) =
\begin{cases}
1 & \text{if  $|x|\leq 1$}\\
g(|x|) & \text{if $<1|x|\leq c_\kappa$}\\
0 & \text{if $|x| > c_\kappa$}
\end{cases}       
\end{equation}
where $|g(x)|<1$.  The most commonly used flat-top kernel is defined by
\begin{equation}
\kappa(x) =
\begin{cases}
1 & \text{if  $|x|\leq 1$}\\
2 - |x|& \text{if $1<|x|\leq 2$}\\
0 & \text{if $|x| > 2$} ;
\end{cases}       
\end{equation}
see \cite{politis1995}. 
Let $l$ be the bandwidth of choice and $\kappa_l(x) = \kappa(x/l)$. The tapered estimator $\widehat{\Sigma}_n = \widehat{\Sigma}^{(\kappa,l)}_n$ at entry $(i,j)$ has value:  

\begin{equation}\label{eq:23}
\widehat{\Sigma}^{(\kappa,l)}_n(i,j) = \kappa_l(k)\left(\frac{1}{n}\sum_{t=1}^{n-k} Z_tZ_{t+k}\right),
\end{equation}
where $|i-j| = k$.
\begin{remark} \rm   
It is worth noting that the above estimator is not guaranteed to be positive definite (PD) for finite samples, but can be corrected towards PD by looking into its Cholesky decomposition; for details see \cite{mcmurrypolitis2010}, \cite{mcmurry2015}. Asymptotically, the estimator is PD with probability tending to one, and the corrected estimator enjoys the same rate of convergence as the original estimator. Either way, it is not   a problem in asymptotic studies. 
\end{remark}
 
While convergence result exist  for $\widehat{\Sigma}_n$, the true series $\{Z_t\}$ is latent and can not be used in the calculation of $\widehat{\Sigma}_n$.
 We can only use the estimator
with the estimated $\tildehat{Z}_t$ defined in Section 2.1 that are calculated from original data $Y_t$, resulting in an estimator $\wwidehat{\Sigma}_n$, where 
 \begin{equation}
	\wwidehat{\Sigma}_n^{(\kappa,l)} (i,j)= \kappa_l(k)\left(\frac{1}{n}\sum_{t=1}^{n-k} \tildehat{Z}_t\tildehat{Z}_{t+k}\right).
 \end{equation}
  Consistency of $\wwidehat{\Sigma}_n$ to $\Sigma_n$ will be shown in the following.

\section{Assumptions and preliminary results}
\subsection{Acceptable forms of $\{Y_t\}$} The stationary series $\{Y_t\}$   in our setting 
will be assumed to have the form:
\begin{equation}\label{eq:31}
	Y_t = f(W_t)
\end{equation}
where $f$ is some continuously differentiable function such that the CDF  of $Y_t$ $F_Y$ is strictly increasing and continuously differentiable, and  $\{W_t\}$ is a strictly stationary Gaussian process.
Without loss of generality, we may assume that $EW_t=0$ since the mean of $W_t$
can be incorporated in the function $f$.

 Equation (\ref{eq:31}) is a common form of extension from Gaussian series to non-Gaussian case. It has been used in the study of long range dependence, as well as analyzing time series with heavy tails;
 see \cite{nonstablegaussian}.
 This assumption also figures in a completely different setting, namely that of Bayesian
machine learning; see \cite{NIPS2003_2481}.

By the Wold decomposition (see \cite{Brockwell1991}) coupled with  Gaussianity, $W_t$ admits the following expansion
\begin{equation}
W_t = \sum_{j=0}^\infty a_j\epsilon_{t-j} + V_t
\end{equation}
where $\epsilon_t\overset{i.i.d.}{\sim}\mathcal{N}(0,1)$, and $V_t$ is a deterministic process independent from $\{\epsilon_t\}$. 
 If we assume $W_t$ is purely nondeterministic, then $V_t$  vanishes, and 
it is clear that $Y_t$ is of the form 
$Y_t = h(\cdots,\epsilon_{t-1},\epsilon_t)$
for some function $h$. Such representation naturally appears in many time series and dynamical system models, and is also a common form used for developing short range dependence conditions;
see \cite{Wu2005b} for details.

   The following lemma holds.
\begin{lemma}\label{lm:31}
Let	$Z_t  = \Phi^{-1}\circ F_Y(Y_t)$. Them $Z_1,\cdots,Z_n $ are jointly normal if and only if $Y_t$ admits the representation in equation (\ref{eq:31}).

\end{lemma}

\begin{proof}
	If (\ref{eq:31}) holds, $Z_t = \Phi^{-1}\circ F_Y\circ f (W_t)$. Since both $W_t$ and $Z_t$ are normally distributed, $\Phi^{-1}\circ F_Y \circ f$ is a normality preserving, continuously differentiable transform. Therefore the transform is linear by Corollary 2 of \cite{mase1977}. 
Therefore, each $Z_t$ is linearly transformed from $W_t$, $t\in\{1,2,\cdots,n\}$. Collecting the n instaneous linear transforms we obtain a transform of the vector $(W_1,\cdots,W_n)$ to $(Z_1,\cdots,Z_n)$ that is linear. Since $(W_1,\cdots,W_n)$ are jointly normal and a linear transform preserves normality, $(Z_1,\cdots,Z_n)$ is multivariate normal; see Lemma 3.1 of  \cite{Das2017}.

Conversely, given $Z_t = \Phi^{-1}\circ F_Y (Y_t)$ is jointly normal and strict stationary, we have $Y_t = F_Y^{-1}\circ\Phi(Z_t)$, which is of the form (\ref{eq:31}).
\end{proof}
The above lemma essentially clarifies the scope of time series models the model-free algorithm can be applied to. Since $F^{-1}_Y\circ \Phi$ is  
monotone, it is reasonable to believe that the   function $f$ is monotone. In fact, $f(\cdot)$ is equivalent to $F^{-1}_Y\circ\Phi(\cdot)$ modulo affine differences in their arguments. Thus we have the following lemma:
\begin{lemma}\label{lm:32}
Under the setup in equation (\ref{eq:31}), given $F_Y$ is strictly increasing, then $f$ is a strictly monotone function.
\end{lemma}

\begin{remark} \rm
 One large subclass of strictly monotone, continuously differentiable functions is an $f$ with derivative bounded away from $0$.  As it turns out, this subclass can simplify our short range dependence assumption to be introduced next.

\end{remark}

\subsection{Short range dependence (SRD) assumptions on $\{Y_t\}$} 
The SRD assumption is necessary for both consistency of mentioned transforms and central limit theorems. Finding the appropriate SRD assumption is a core problem  for proving bootstrap validity in our setting since
direct calculation of bootstrap variances is almost impossible.  We must take into consideration the special characteristics of both our bootstrap procedure and bootstrap pseudo data and find the SRD assumption that can be well incorporated in the proofs. Thus the SRD condition should fulfill the requirement that with it assumed:
\begin{enumerate}
\item Consistency of proposed transforms can be established in proper mathematical forms;
\item Certain central limit theorems can be established.
\item Consistency of bootstrap variance, as further explained in Section 4.
\end{enumerate}

 One of the most widely used SRD assumptions for strictly stationary time series is the strong mixing condition introduced by \cite{Rosenblatt43}, and extensively studied since. Many useful results are available for series under certain mixing rates. However such conditions are often hard to verify for general time series models. 
More recently, \cite{Wu2005b} introduced the 
physical dependence measure described in the following. Assume $Y_t = h(\cdots,\epsilon_{t-1},\epsilon_t)\in L^p$ with $\epsilon_t\overset{i.i.d.}{\sim} F_\epsilon$; then, the physical dependence measure is defined by:
\begin{equation}
	\delta_p(j) = \norm{Y_j - Y_j'}_p
\end{equation}
where $Y_j' = h(\cdots,\epsilon_{-1},\epsilon_0',\cdots,\epsilon_j)$ with $\epsilon_0'\sim F_\epsilon$  independent of $\{\epsilon_t\}$. The series $Y_t$ is called strongly $p-$ stable if
\begin{equation}
(C0)\quad\Delta_p =  \sum_{j=0}^\infty \delta_p(j) <\infty.
\end{equation}

The $p$-stable assumption was able to provide the first results on the convergence of banded
and/or tapered autocovariance matrix estimators by \cite{Wu2009} 
and  \cite{mcmurrypolitis2010}.
Yet results for functional central limit theorem for the empirical process--which relates to uniform consistency of empirical CDF--can not be readily established with this assumption,  and require  further conditions. To simplify assumptions, we hope for a condition that offers more flexibility than what is mentioned above. One short range dependence measure that is gaining interest is the $m-approximation$ assumption developed in a series of papers in \cite{BERKES20091298}, \cite{berkes2011} and  \cite{hormann2010}.

The following related conditions are used for proving different properties and are very important in our proofs.
\theoremstyle{definition}
\begin{definition}
The following different $m-approximation$ conditions are proposed. Let $\{Y_t\}$ be a strictly stationary time series,
there exists  an $m-$ dependent series $\{
Y_t^{(m)}\}$ that
\begin{itquote}

(C1) $\forall t\in\mathbf{Z}$,  $P(|Y_t -Y_t^{(m)}|>\gamma_m)<\delta_m$,  for some sequence $\gamma_m\rightarrow 0$ and $\delta_m\rightarrow0$.\\

(C2) both $Y_{t}$ and $Y_t^{(m)}$ are in  $L^p$;  $\exists \delta(m):\mathbb{N}\rightarrow\mathbb{R}_+$,
satisfying $\delta(m)\ll m^{-A}$ for some $A>0$, and such that $\norm{Y_t - Y_t^{(m)}}_p\leq\delta(m)$. \\

(C3) both $Y_{t}$ and $Y_t^{(m)}$ are in  $L^p$; $\sum_{m = 0}^{\infty} \norm{Y_m - Y_m^{(m)}}_p<\infty$.\\
\end{itquote}
Here, the notation  "$a_n\ll b_n$" means $\limsup{\left\lvert\frac{a_n}{b_n}\right\rvert}\rightarrow 0$.

\end{definition}

Assumption (C1) was introduced in \cite{BERKES20091298} to establish asymptotic behavior of  empirical process. (C2) was introduced in \cite{berkes2011} to establish invariance principle for partial sums. Both are relevant in our setting. (C3) appeared in \cite{hormann2010} in the context of functional time series. While it is not directly related to the setup here, there is an important relation between (C3) and the physical dependence measure (C0), under which consistency of tapered autocovariance matrix estimator can be established.

There are various methods to construct the $m$-dependent sequence $Y_t^{(m)}$. Typically one can use truncation, substitution and coupling method on the representation $Y_t = h(\cdots,\epsilon_{t-1},\epsilon_t)$; for details refer to \cite{BERKES20091298}. The coupling construction is most desirable: we replace $\epsilon$ by i.i.d. independent copies of $\epsilon'$ for times that are at least $m$ steps away from current time $t$, i.e. let $Y_t^{(m)} = h(\cdots,\epsilon_{t-m-1}',\epsilon_{t-m}',
\epsilon_{t-m+1},\cdots,\epsilon_t)$. The resulting $Y_t^{(m)}$ is $m-$ dependent, and also has the advantage that it has the same distribution and moments as $Y_t$.    This construction along with (C2) were also used in \cite{wu2005a}
where it was called the geometric-moment contracting property, assuming a faster geometric decay rate.  We will use it in what follows. 

The above SRD conditions are related;   the following lemma clarifies. 
\begin{lemma}\label{lm:33}
	Suppose that $\{Y_t^{(m)}\}$ is constructed by coupling as defined above, then: 
\quad	\\

1. Assume (C2)  with $A =C(\frac{1}{p} + \frac{1}{\theta})$ for some $C>0$, $\theta\in(0,1)$; then (C1) holds with $$P(|Y_t - Y_t^{(m)}|>m^{-C/\theta})<m^{-C}.$$\\

2. Assume (C2) with $A>1$; then (C3) holds.\\

3. Assume (C3); then (C0) holds.\\

4. (C2) is preserved (with a new rate) under  $\theta$-Lipschitz transforms. To elaborate, 
let $g$ be a $\theta-$Lipschitz function, i.e., for some constant $K$ and $\theta\in(0,1]$, $|g(x) - g(y)|\leq K|x-y|^\theta$. If $\norm{Y_t - Y_t^{(m)}}_p\leq \delta(m)$ with $\delta(m)\ll m^{-A}$, then 
$$\norm{g(Y_t) - g(Y_t^{(m)})}_p \leq K \delta(m)^\theta\ll m^{-\theta A}.$$
\end{lemma}
\begin{proof}
See Appendix.
\end{proof}
Based on the above lemma, it is convenient that we assume condition (C2) with appropriate rate since other dependence measures can be derived from it.

 We summarize our first assumptions as follows. \\

\begin{assumption}
	\quad \\
	
(A1) $Y_t = f(W_t)$, where $f: \mathbb{R} \rightarrow \mathbb{R}$ is a continuously differentiable function, and $W_t$ is a (zero mean) stationary Gaussian process with spectral density bounded and strictly bounded away from 0. 
Also assume $W_t$ is purely nondeterministic.\\

(A2) The CDF $F_Y(\cdot)$ is strictly increasing, $\theta$-Lipschitz and continuously differentiable with density function $f_Y(\cdot)>0$.\\

(A3) $Y_t\in L^p$ satisfies (C2) with $p>2$, and $A$ to be specified later. Also, $\exists c>0$, $|f'|\geq c >0$.
 \\

Or:\\

(A4) $W_t$ satisfies (C2)  with $p>2$, and $A$ to be specified later. Also,  $f$ preserves the  (C2) property.
\end{assumption}

\begin{remark} \rm 
	(a). It is possible to relax (A1) and (A2) when 
 $f$ is a possibly discontinous but strictly monotone function, and $F_Y$ is an absolutely continuous function subject to some extra constraints. It can be shown that Lemma \ref{lm:31} still holds under this assumption;
see Corollary 3.1 of \cite{Das2017}. Our simulations in Section 6 also confirms this.  However for the purpose of analysis and to avoid edge results, we require a stronger assumption as in (A1) and (A2).

	(b). Assumptions (A3) and (A4) have the same purpose: to establish SRD for  $\{Y_t\}$ and $\{Z_t\}$. It turns out with the derivative of the transfer function $f$ strictly bounded away from 0, SRD of $Y_t$ will deduce SRD for the other series; while otherwise SRD assumption on both $W_t$ and $Y_t$ is required. This is because it is not clear yet what kind of functions $f$ preserve the  $m-$approximable property.
	
	(c). Under assumptions (A1) and (A2), by Lemma \ref{lm:32}, $f$ is also a strictly monotone function. 
\end{remark}
\begin{lemma}
Assume (A1)--(A3), then  $\{W_t\}$ and $\{Z_t\}$ satisfy (C2). 
\end{lemma}\label{lm:34}
\begin{proof}
Under (A1), (A2) and by Lemma \ref{lm:32}, $f$ is continuously differentiable and strictly monotone. Therefore it is invertible and $|(f^{-1})'| = |\frac{1}{f'}|$ is bounded, which means $f^{-1}$ is a Lipschitz function. Since (C2) is preserved under Lipschitz transform and $W_t = f^{-1}(Y_t)$, $\{W_t\}$ also satisfies (C2) with same rate $\delta(m)$ as  $\{Y_t\}$.  By Lemma  3.1,  the transform from $W_t$ to $Z_t$ is linear; thus $\{Z_t\}$ also satisfies (C2) with same rate $\delta(m)$.\\
\end{proof}
\subsection{Preliminary results}
Now we quote the necessary theorems relative to our work.
 \begin{theorem}\label{th:35}(Berkes, H\"ormann, Schauer (2009))
\\
	Define $R(s,t) = \sum_{1\leq t\leq n} (I\{Y_k\leq s\} - F_Y(s))$.
	Assume $F_Y$ is $\theta$-Lipschitz continuous, and (C1) holds with $\gamma_m = m^{-C/\theta}$,
 $\delta_m = m^{-C}$ and some $C>4$. Then there exists a two parameter Gaussian process $K(s,t)$ with $E(K(s,t)K(s',t')) = (t\wedge t')\Gamma(s,s')$, such that
	\begin{equation}
	\sup_{s,t} |R(s,t) - K(s,t)| = o(n^{1/2}(\log n)^{-\alpha} ), \ a.s.
	\end{equation}
	for some $\alpha>0$. In addition,
	\begin{equation}\label{eq:36}
\Gamma(s,s') = \sum_{k\in\mathbb{Z}}E(I(Y_0\leq s) - F_Y(s))(I(Y_k\leq s') - F_Y(s'))
	\end{equation}
	 is absolutely convergent for all choices of $s,s'$.
\end{theorem}

\begin{remark} \rm 
	We are more interested in the case $t = n$ such that the above theorem reduces to a 1-dimensional centered Gaussian process $K$ indexed by $s\in \mathbb{R}$. It turns out that tightness of $K$ plays an important role in the accuracy of the estimated $\widehat{U}_t$. Although Theorem \ref{th:35} does not guarantee the limiting Gaussian process to be tight, it is reasonable to believe so and also it can be shown that for series of the form (\ref{eq:31}) tightness can be shown through continuity and boundedness of the covariance structure.
\end{remark}

	 \begin{theorem}\label{th:36}(Berkes, H\"ormann, Schauer (2011))
	 	\\
		Assume (C2) with $p >2$, $\eta \in (0,1)$,  $A >\frac{p-2}{2\eta}(1-\frac{1+\eta}{p})\vee 1$.  Then $\sigma_{\infty}^2 = \sum_{k\in\mathbb{Z}} \gamma_Y(k)$ is absolutely convergent. Also,
		\begin{equation}
		\sum_{k=1}^n \left(Y_k -  EY_0\right) =  W_1(s_n^2) + W_2(t_n^2) + O(n^{(1+\eta)/p}),  \ a.s.
		\end{equation}
where $W_1$ and $W_2$ are two Brownian motions and $s_n^2\sim \sigma_\infty^2n$, $t_n^2\sim cn^{\gamma}$ with $\gamma\in (0,1)$. Here $a_n\sim b_n$ means $\lim_{n\rightarrow\infty} \frac{a_n}{b_n}= 1$; $a\vee b = \max\{a,b\}$.
	\end{theorem}

 \begin{theorem}\label{th:37}(McMurry \& Politis(2010))
 	\\
	For the tapered estimator defined in equation (\ref{eq:23}) with tapering function (\ref{eq:21}), assume (C0) with $2<p\leq4$ for $\{Z_t\}$ and $\sum_{k\in \mathbb{Z}} \sigma(k)<\infty$, with $l = o (n^{\frac{p-2}{p}})$ then, 
	\begin{equation}
\sum_{k=0}^{\lfloor c_{\kappa}l\rfloor}|\widehat{\sigma}(k) - \sigma(k)|\underset{L^{p/2}}{\rightarrow} 0,
	\end{equation}
	and
	\begin{equation}
\norm{\widehat{\Sigma}_n^{(\kappa,l)}- \Sigma_n}_{op} \underset{L^{p/2}}{\rightarrow} 0.
	\end{equation}
	Furthermore, assuming $Z_t$ has spectral density $f_Z$ satisfying $0<c_1 \leq f_Z(w) \leq c_2<\infty, \forall w\in [0,2\pi]$, $\widehat{\Sigma}_n^{(\kappa,l)}$ is positive definite with probability tending to 1, and 
	\begin{equation}
\norm{(\widehat{\Sigma}_n^{(\kappa,l)})^{-1}- \Sigma_n^{-1}}_{op} = o_p(1) 
	\end{equation}
	$$$$
\end{theorem}
We now proceed  to the proof of bootstrap validity for parameter estimation and prediction. 

\section{Model-free bootstrap validity for estimation}
For the following paragraphs, let $P^*$ denote the probability in the bootstrap world. To show that  \textbf{Algorithm 1} produces asymptotically correct bootstrap confidence interval,  the following type of results will be shown in this section:
\begin{equation}\label{eq:41}
	\sup_{x\in\mathbb{R}} |P^*(\tau_n(\widehat{\theta}^*_n - \widehat{\theta}_n)\leq x) - P(\tau_n(\widehat{\theta}_n - \theta_0)\leq x)| \xrightarrow{n\rightarrow\infty}
	  0,\: in\: probability
\end{equation}
Moving forward all convergence is with respect to $n\rightarrow\infty$ and this notation will be omitted for simplicity.
$\tau_n$ is the   rate of convergence for the estimator $\widehat\theta_n$. 
 In many circumstances, $\tau_n = \sqrt{n}$ and $\widehat\theta_n$ is a $\sqrt{n}-$consistent estimator for $\theta_0$; 
 the most important such case is the mean where $\theta_0 = E(Y_0)$
and $\widehat{\theta}_n = \frac{1}{n}\sum_{t=1}^n Y_t$. Moreover, the class of functions of linear statistics is most useful, namely when 
the statistic can be written as 
\begin{equation}\label{eq:42}
 q(\frac{1}{n-k+1}\sum_{t=1}^{n-k+1} g(Y_t,\cdots,Y_{t+k-1}))
\end{equation}
 where $q: \mathbb{R}^q\rightarrow\mathbb{R}^{\tilde{q}}$ and $g: \mathbb{R}^k \rightarrow \mathbb{R}^q$ are smooth functions.
By the $\delta$-method, the statistic (\ref{eq:42}) inherits the 
 $\sqrt{n}-$ consistency of the linear statistic. The class of statistics of the type (\ref{eq:42})  
includes a wide range of estimators in time series, e.g. sample mean, sample autocovariance and sample autocorrelation, to name a few. The class  (\ref{eq:42})  was   investigated by \cite{kunsch1989}, 
\cite{politis1992}, \cite{buhlmann1997},\cite{Lahiri2003} and \cite{kreiss2011}.

 \subsection{Bootstrap validity for the mean}
 In this section, we focus on model-free bootstrap for the mean $\theta_0 = E(Y_0)$ and  extension to statistic of the form (\ref{eq:42}) will be addressed in later section. Therefore we let $\tau_n = \sqrt{n}$ for later discussions. Typically, convergence of the type (\ref{eq:41}) is proved in a two step procedure:
 \begin{enumerate}
 	\item Asymptotic normality for the estimator and its bootstrapped analogue : 
 	\begin{equation}
 	d_\infty(\sqrt{n}(\widehat{\theta}_n - \theta_0),\mathcal{N}(0,\sigma^2))\rightarrow 0, \sigma^2 < \infty. 
 	\end{equation}
 	\begin{equation}
d_\infty(\sqrt{n}(\widehat{\theta}^*_n - E^*(\widehat{\theta}^*_n)), \mathcal{N}(0,(\sigma^*_n)^2))\rightarrow 0, in \: probability
 	\end{equation}
 	where $d_\infty(\cdot,\cdot)$ denotes the Kolmogorov metric between two (one dimensional) distributions defined by: $$d_\infty(F,G) = \sup_{x\in\mathbb{R}}|F(x) - G(x)|.$$

and where $\sigma^2$ and $(\sigma_n^*)^2$ are the asymptotic variances of the centered and $\sqrt{n}$-scaled estimator $\sqrt{n}(\widehat{\theta}_n - \theta_0)$ in the real world and $\sqrt{n}(\widehat{\theta}^*_n - E^*(\widehat{\theta}^*_n))$ in the bootstrap world, respectively.
 	\item Asymptotically equal variances:
 	\begin{equation}
		(\sigma^*_n)^2 \overset{P}{\rightarrow} \sigma^2.
 	\end{equation}
 \end{enumerate}
\begin{remark} \rm 
	It is believed that asymptotic normality is often necessary for the asymptotic validity of bootstrap for general inference problems, and thus $\tau_n = \sqrt{n}$ is also necessary. For a detailed explanation, see Ch. 6 of \cite{Dehling2002}.
\end{remark}

 Following the two step procedure, we have the following:

\begin{lemma}\label{th:41}
	Let $\{Y_t\}$ satisfy assumptions (A1)-(A4) with $p >2$, $\eta \in (0,1)$, $A >\frac{p-2}{2\eta}(1-\frac{1+\eta}{p})\vee 1$ and $(1 + \eta)/p - 1/2 <0$. Then  
	$\gamma_Y(k)$ is absolutely summable.
	Also, 
	\begin{equation}
		\frac{1}{\sqrt{n}}\sum_{k=1}^n \left(Y_k -  EY_0\right) \overset{d}{\rightarrow}  \mathcal{N}(0,\sigma_\infty^2)
	\end{equation}
	with $\sigma_\infty^2 = \sum_{k\in\mathbb{Z}}\gamma_Y(k)$.
\end{lemma}

\begin{proof}
	 By Theorem \ref{th:35}, $\gamma_Y(k)$ is absolutely summable for $k\in\mathbb{Z}$. Then
	\begin{equation}
	\begin{split}
	\sqrt{n}(\bar{Y}_n - EY_0) &= \frac{1}{\sqrt{n}}\left( W_1(s_n^2) + W_2(t_n^2) + O_{ a.s.}(n^{(1+\eta)/p})\right)\\
	& = W_1(s_n^2/n) + W_2(t_n^2/n) + O_{ a.s.}(n^{(1+\eta)/p - 1/2}) \\
	&\overset{d}{\rightarrow} \mathcal{N}(0,\sigma^2_\infty).
	\end{split}
	\end{equation}
	since  $n^{\gamma -1}\rightarrow 0$ and  $n^{(1+\eta)/p - 1/2} \rightarrow0$.
	
\end{proof}

\begin{lemma}\label{th:42}
	Under the assumptions of Lemma \ref{th:41}, and the
additional assumption $\frac{A}{\frac{1}{p} + \frac{1}{\theta}}>4$, we have 
	\begin{equation}\label{eq:49}
			P(|Y_t - Y_t^{(m)}|>m^{-C/\theta})<m^{-C}
	\end{equation}
	with $C>4$, and
	\begin{equation}
	\sup_{t\in[n]}|\widehat{U}_t - U_t| = O_p(\frac{1}{\sqrt{n}})
	\end{equation}
\end{lemma}
\begin{proof}
See Appendix.
\end{proof}
\begin{remark} \rm
	The above lemma argues that the uniform error for the first step transform $\widehat U_t = F_Y(Y_t)$
 is of $O_p(\frac{1}{\sqrt{n}})$ rate. This is important as consistency of $\wwidehat{\Sigma}_n$ to $\Sigma_n$ depends on the consistency of the estimated $\tildehat{Z}_t$ to $Z_t$, which itself depends on the consistency of $\widehat{U}_t$ to $U_t$. 
\end{remark}
To achieve the same result as in Lemma \ref{th:42} with kernel estimators $\widehat{F}_h$, we assume the following:
\begin{assumption}
		(A5) $F_Y$ is $2^{nd}$ order continuously differentiable such that 

		$F_Y^{(2)}(x) =\frac{d^2 F_Y(x)}{dx^2}$ is bounded and H\"older continuous. \\
		(A6) The kernel $K$ is a distribution function with density $k$ satisfying:
		\begin{enumerate}
				\item $k(x)$ is bounded.
				\item $k(x) = k(-x)$.
				\item $\int |x|^pk(x) dx<\infty$.
		\end{enumerate}
		(A7) The class of functions $\mathcal{F}_{K} = \{K(x_0 - x): x_0\in \mathbb{R}\}$ is Donsker with respect to $\{Y_t\}$ under condition (C1).  In other words, 
		there exists a tight Gaussian process $G_h(x), x\in\mathbb{R}$, such that
				\begin{equation}
			\sup_{x\in\mathbb{R}} \left\lvert \frac{1}{\sqrt{n}}\sum_{k=1}^n \left(K(\frac{x - Y_k}{h}) - E\left(K(\frac{x - Y_0}{h})\right) \right)- G_h(x)\right\rvert = o_p(n^{1/2}).
		\end{equation}
		In the above, the covariance of $G_h(x)$ is given by
 $$\Gamma_{G_h}(x,x') = \sum_{k\in \mathbb{Z}}E\left(K(\frac{x-Y_0}{h}) - E(K(\frac{x-Y_0}{h}))\right)\left(K(\frac{x' - Y_k}{h}) - E(K(\frac{x'-Y_k}{h}))\right).$$
		\\
		(A8) $h = o(n^{-1/4})$.
\end{assumption}
\begin{remark}\rm
	 The purpose of assumption (A7) is to help establish tightness for the supremum of the empirical process, which helps establish the uniform $O_p(\frac{1}{\sqrt n})$ error rate for the first transform. It is not yet clear if the $m-$approximable condition can lead to (A7); however under similar dependence conditions, (A7) is shown to be correct. For example Theorem 2.1 in \cite{Arcones1994} established such a result for $\beta-$mixing stationary series with $\mathcal{F}_K$ a V-C subgraph class of functions.
\end{remark} 
\begin{lemma}\label{th:43}
Let $\widehat{U}^{(h)}_t = \widehat F_h(Y_t)$.	With assumptions (A1)--(A8) holding,  
		$$\sup_{t\in[n]}|\widehat{U}_t^{(h)} - U_t| = O_p(\frac{1}{\sqrt{n}}).$$
\end{lemma}
\begin{proof}
	See Appendix.
\end{proof}

With assumptions (A1)-(A8) along with Lemma \ref{th:42} and \ref{th:43} holding, we can prove 
 MFB validity results using either the empirical CDF estimator $\bar F$ or the kernel CDF estimator $\widehat F_h$. Therefore in the following results, we only consider a general CDF estimator $\widehat F$ that can represent either $\bar F$ or $\widehat F_h$.
\begin{theorem}\label{th:44}
	Assume the  conditions of Lemma \ref{th:42} or  \ref{th:43}, with appropriate rate for $l(n) \rightarrow\infty$ and $c(n)\rightarrow\infty$. Then  
	\begin{equation}\label{eq:413}
		\sum_{k=0}^{\lfloor c_{\kappa}l\rfloor}|\wwidehat{\sigma}(k) - \sigma(k)|\overset{P}{\rightarrow} 0
	\end{equation}
	which implies 
	\begin{equation}\label{eq:414}
			\norm{\wwidehat{\Sigma}_n - {\Sigma}_n}_{op}\overset{P}{\rightarrow}0
	\end{equation}
and
	\begin{equation}\label{eq:415}
\norm{\wwidehat{\Sigma}_n^{-1} - {\Sigma}_n^{-1}}_{op}\overset{P}{\rightarrow}0.
\end{equation}
\end{theorem}

\begin{proof} 
	See Appendix.
\end{proof}
With the above results, we can show that the pseudo data generated by MF and LMF bootstrap procedure converge in distribution to the true distribution, in probability.
Before proceeding, we state one more assumption:

\begin{assumption}
(A9) $\norm{\wwidehat{\Sigma}_n^{1/2} - \Sigma_n^{1/2}}_{op} = o_p(1)$ and $\norm{\wwidehat{\Sigma}_n^{-1/2} - \Sigma_n^{-1/2}}_{op} = o_p(1)$.
\end{assumption}
Assumption (A9) holds if we have an additional mild convergence rate for the estimator towards the true covariance matrix, which can be achieved by choosing appropriate tapering parameter $l(n)\rightarrow\infty$ and normal thresholding parameter $c(n)\rightarrow\infty$.  By Theorem 2.1 of \cite{Drmac:1994:PCF:196045.196081}, if 
\begin{equation}\label{eq:419}
(\log n)^2 \norm{\wwidehat{\Sigma}_n - \Sigma_n}_{op} \overset{P}{\rightarrow} 0
\end{equation}
then (A9) holds; see also \cite{Jentsch} and the Corrigendum of \cite{mcmurrypolitis2010}).

\begin{lemma}\label{lm:45}
	Assume (A1)-(A8).
	(1) For LMF bootstrap, $\forall d\in\mathbb{N}$,
	\begin{equation}\label{eq:417}
	(Y_{t_1}^*,Y_{t_2}^*,\cdots,Y_{t_d}^*)\overset{d^*}{\rightarrow} (Y_{t_1},Y_{t_2},\cdots,Y_{t_d}), \text{in probability}.
	\end{equation}

	(2) Further assume  (A9).
	Then the model-free procedure is asympototically equivalent to limit model-free procedure. i.e.,
the infinite sequence 
	\begin{equation}\label{eq:418}
	(\xi_1^*,\xi_2^*,\cdots)\text{ converges  in  distribution to }
(\zeta_1, \zeta_2,\cdots)
 	\text{  in probability}  
	\end{equation}
where $(\zeta_1, \zeta_2,\cdots)$ is an infinite sequence
with entries being i.i.d.~$\mathcal{N}(0,1)$. Furthermore, 
  equation (\ref{eq:417}) holds for the model-free bootstrap.
\end{lemma}

\begin{proof}
	See Appendix.
\end{proof}

Let  $\gamma_{Y}^*(k)$ and $\Gamma_n^*$ denote the bootstrap analogues of $\gamma_Y(k)$
and $\Gamma_{n}$. The long-run variance in the bootstrap world
is then  $(\sigma_\infty^*)^2 =  \sum_{k\in\mathbb{Z}}\gamma_Y^*(k)$.


\begin{theorem}\label{th:46}
Under (A1)-(A9). For both MF bootstrap and LMF bootstrap
\begin{equation}
(\sigma_\infty^*)^2 \rightarrow \sigma_\infty^2, in\: probability.
\end{equation}
and
\begin{equation}
\sup_{x\in \mathbb{R}} \left\lvert P^*(\sqrt{n}(\bar{Y}_n^* - E^*(Y_t^*))\leq x) - P(\sqrt{n}(\bar{Y}_n - E(Y_0))\leq x)\right\rvert\overset{P}{\rightarrow} 0.
\end{equation}

\end{theorem}
\begin{proof}
	See Appendix.
\end{proof}

\subsection{Bootstrap validity for smooth functions of linear statistics}
\label{section:vsfls}

Now that the model-free bootstrap validity for the mean has been shown, we can extend the result for statistics of the form (\ref{eq:42}). For convenience, denote $X_t  = (Y_t,\cdots,Y_{t+k-1})$, 
 $\widehat{ \theta}_n = \frac{1}{n-k+1} \sum_{t=1}^{n-k+1} g(X_t)$ and $\theta_0 = E(g(X_t))$ with bootstrap analogue $\widehat{ \theta}_n^* = \frac{1}{n-k+1} \sum_{t=1}^{n-k+1} g(X_t^*)$ and $\theta^*_0 = E^*(g(X_t^*))$, respectively. In order to extend previous results using same proof strategy, we need assumptions such that the key points in the proof are checked:
\begin{assumption}
	\quad\\
	
	(A10) $q(\mathbf{x}):\mathbb{R}^q\rightarrow\mathbb{R}^{\tilde{q}}$ has continuous partial derivatives in a neighborhood of $\theta_0$, and $\sum_{i=1}^q(\partial q/\partial x_i)|_{x = \theta_0} x_i$ does not vanish. \\
	
	(A11) $g = (g_1,\cdots,g_q)$ is a continuous function and $\forall i\leq q$, $g_i(X_t)$ is $L^p-m-$approximable in the sense of (C2) with $p>2$ and $g_i(X_t^{(m)})$ its $m-$approximation, and appropriate constant $A$ such that Lemma  \ref{th:41} holds for all $g_i(X_t)$.
	
\end{assumption}
\begin{theorem}\label{th:48}
	With (A1)- (A11) hold, for MF bootstrap and LMF bootstrap with empirical CDF $\bar F$ and kernel smoothed CDF $\widehat F_h$,
	\begin{equation}
				\sup_{x\in\mathbb{R}^{\tilde{q}}} |P^*(\sqrt n(q(\widehat{\theta}^*_n) - q(\theta_0^*))\leq x) - P(\sqrt n(q(\widehat{\theta}_n) - q(\theta_0))\leq x)| \overset{P}{\rightarrow} 0.
	\end{equation}
\end{theorem}
The proof is similar to Theorem \ref{th:46} by checking   the key
 points listed there for $g(X_t)$ in combination with Cram\'er-Wold device and the $\delta$-method, therefore is omitted.
The mean is now a special case of Theorem \ref{th:48}. Moreover, it can be applied to more general statistics such as  autocovariance and/or autocorrelation.  We will address the autocovariance case in the following.

\subsubsection{Autocovariances} For simplicity, let $X_t = (Y_t,Y_{t+1},\cdots,Y_{t+k})\in \mathbb{R}^{k+1}$, $g(X_t) = ( Y_tY_{t+1}, \cdots,Y_tY_{t+k})$ and $q$ is the identity map. Then $\frac{1}{n-k}\sum_{t=1}^{n-k} g(X_t)$ is a version of an autocovariance estimator that estimates up to lag $k$ autocovariances, assuming $EY_0 = 0$ . Note that we can also design $g$ and $q$ more carefully such that we get the usual autocovariance estimators. Then the following corollary holds:
 \begin{corollary}
Assume $Y_t\in L^p$ $p>4$ such that (C2) is satisfied with appropriate rate $A$, then $g(X_t)$ as defined above also satisfies (C2) with $p' = \frac{p}{2}>2$. Then by Theorem \ref{th:48} the model-free bootstrap procedure is asymptotically correct for bootstrapping vector of autocovariances.

 \end{corollary}

\begin{proof}
	Let $Y_t^{(m)}$ be the approximation series generated by coupling. Then 
	\begin{equation}
	\begin{split}
		\norm{Y_tY_{t+k} - Y_t^{(m)}Y_{t+k}^{(m)}}_{p'}
		 & = 
			\norm{Y_tY_{t+k} - Y_t^{(m)}Y_{t+k} + Y_t^{(m)}Y_{t+k} -Y_t^{(m)}Y_{t+k}^{(m)}}_{p'}\\
		&\leq
			\norm{Y_{t+k}}_{p}\norm{Y_t-Y_t^{(m)}}_{p} + 
			\norm{Y_{t}^{(m)}}_{p}\norm{Y_{t+k}-Y_{t+k}^{(m)}}_{p} \\
		&=
		2\norm{Y_t}_{p}\norm{Y_t-Y_t^{(m)}}_{p}\\
		&= o(m^{-A})
	\end{split}
	\end{equation}
	To complete proof, we apply Theorem \ref{th:48}.
\end{proof}
\begin{remark}\rm 
	(a). The usual assumption of finite sum of $4^{th}$ order cumulants of $Y_t$ that ensures finiteness of asymptotic variance  can be dropped since it can be deduced by  checking assumption (A11) along with Lemma \ref{th:41}.
	
	(b). Notably, \cite{kreiss2011} showed that the autoregressive(AR) sieve bootstrap procedure does not work in general for bootstrapping autocovariances of strictly stationary series (even for linear
series that are 
not Gaussian) because of inconsistency of the limiting variance associated the companion AR process. In essence, AR-sieve bootstrap (and the previously mentioned
Linear Process bootstrap) mimic correctly the first and second order moment structure of 
$Y_t$; if the statistic of interest has a large-sample distribution that depends on 
higher order moments, the AR-sieve bootstrap (and the LPB) 
 may fail. However, bootstrap validity holds for the MFB  due to the
 assumption that $Y_t = f(W_t)$ where $W_t$ is a Gaussian process whose  covariance structure can be consistently estimated (up to affine transform). Since a Gaussian process is fully described by its second order moment property, consistency of higher moments can also be obtained by the bootstrap procedure. Therefore the model-free bootstrap also works for higher order statistics.
	
\end{remark}

\subsection{Approximately linear statistics and sample quantiles}
 
Now consider a statistic $\widehat{ \eta}_n$ that can be expressed as
\begin{equation}\label{eq:4111}
\widehat{ \eta}_n = \frac{1}{n } \sum_{t=1}^{n } g(Y_t) +o_p(1/\sqrt{n}) .
\end{equation}
Under   aforementioned conditions,  the model-free bootstrap can be applied
to the linear statistic  $\widehat{ \theta}_n = \frac{1}{n } \sum_{t=1}^{n } g(Y_t)$.
Since $\widehat{ \theta}_n$   will be $\sqrt{n}$-consistent (under the required
conditions), it follows that $\widehat{ \eta}_n$ and $\widehat{ \theta}_n$
are asymptotically equivalent, i.e., 
$\sqrt{n} (\widehat{ \eta}_n-\theta_0) $ has the same asymptotic 
distribution with $\sqrt{n} (\widehat{ \theta}_n-\theta_0) $  where
$\theta_0=E(g(Y_t))$. Therefore, since  the model-free bootstrap works to 
approximate the distribution of $\sqrt{n} (\widehat{ \theta}_n-\theta_0) $, 
it  will also work  to 
approximate the distribution of $\sqrt{n} (\widehat{ \eta}_n-\theta_0) $.

Focusing on sample quantiles, we have the following result:
\begin{lemma}\label{lm:48}
Assume (C2) for $\{Y_t\}$ with appropriate rate $A$ as well as (A2).  Let $u_{n,p} = \bar{F}^{-1}(p)$; $u_p = F_Y^{-1}(p)$. Then for the Bahadur-Kiefer process mentioned in \cite{wu2005a}:
\begin{equation}
\alpha_n = \left\lvert f_Y(u_p)(u_{n,p} - u_p) - (p - \bar{F}(u_p))\right\rvert .
\end{equation}
We have $\forall p\in (0,1)$, $\alpha_n = O_p(r_n)$ where $r_n = o(n^{-1/2})$. As a result,
\begin{equation}\label{eq:4222}
	u_{n,p} - u_p = \frac{p - \bar{F}(u_p)}{f_Y(u_p)} + o_p(n^{-1/2}).
\end{equation}
Under additional assumptions of (A5)-(A8), then the above equation also holds for $u_{n,p} = \widehat{F}_h^{-1}(p)$, i.e.
\begin{equation}\label{eq:4223}
u_{n,p} - u_p = \frac{p - \widehat{F}_h(u_p)}{f_Y(u_p)} + o_p(n^{-1/2}).
\end{equation}
\end{lemma}
\begin{remark}\rm
As shown in \cite{wu2005a},  for the empirical CDF estimator $\bar F$, a stronger version of equation (\ref{eq:4222}) holds provided a faster geometric convergence rate for the $m-$ approximation assumption (C2). 
\end{remark}
In equation (\ref{eq:4222}) and  (\ref{eq:4223}), $\frac{p - \widehat{F}(u_p)}{f_Y(u_p)}$ is of the linear form in equation (\ref{eq:4111}). For equation (\ref{eq:4222}),
$$\frac{p - \bar{F}(u_p)}{f_Y(u_p)} =\frac{1}{n} \sum_{k=1}^n \frac{F_Y(u_p) - I\{Y_k\leq u_p\}}{f_Y(u_p)};$$
the same holds for equation (\ref{eq:4223}) by the relation
$$\frac{p - \widehat{F}_h(u_p)}{f_Y(u_p)} =\frac{1}{n} \sum_{k=1}^n \frac{F_Y(u_p) - K(\frac{u_p - Y_k}{h})}{f_Y(u_p)}.$$
by previous analysis, we have the following theorem.
\begin{theorem}
	Under assumptions of Theorem \ref{th:48} with respect to the linear part $\frac{p - \widehat{F}(u_p)}{f_Y(u_p)}$ and with Lemma \ref{lm:48} holding,  both MF bootstrap and LMF bootstrap are asymptotically valid for sample quantiles.
\end{theorem}
 A special case of interest   is to let  $p = \frac{1}{2}$.  Then, the above discussion shows  validity of the model-free bootstrap for the sample median.

\subsection{Bootstrap validity for kernel smoothed spectral density}
Another important parameter of interest is the spectral density evaluated at some frequency $\omega$.  For this subsection, let $f_{sp.d}(\omega)$ denote the spectral density of $\{Y_t\}$ which is defined as $f_{sp.d}(\omega) = \frac{1}{2\pi}\sum_{k\in\mathbb{Z}} \gamma_Y(k)e^{-ik\omega}$.
The kernel smoothed estimator is a widely used estimator of $f_{sp.d}(\omega)$
 in this setting. Let $I_n(\omega_j) = \frac{1}{2\pi n} \left\lvert\sum Y_k e^{-ik\omega_j}\right\rvert^2$ denote the periodogram of $\{Y_t\}$ where $\omega_j \in \{\frac{2\pi j }{n}, j\in [n]\}$ are the Fourier frequencies. The kernel smoothed spectral density estimator is defined by:
\begin{equation}
		\widehat f_{sp.d}(\omega) = \sum_{j\in [n]} \tilde\kappa_h(\omega - \omega_j) I_n(\omega_j).
\end{equation}
where $\tilde\kappa(\omega)$ is a kernel function under certain assumptions and $\tilde\kappa_h(\cdot) = h^{-1}\tilde\kappa(\cdot/h)$.
Let $\widehat f_{sp.d}^*(\omega)$ be the kernel smoothed estimator based on the bootstrap samples $\{Y_t^*\}_{t=1}^n$ generated through model-free bootstrap. We would like to show validity for this procedure. Besides key assumptions for previous theorems to hold, we also need the following:
\begin{assumption}

		(A12) $\{Y_t\}$ satisfies (C2) with $p=4$; $\inf_{\omega\in [-\pi,\pi]}f_{sp.d}(\omega) >0$.\\
		
		(A13) The kernel $\tilde\kappa(\cdot)$ is a symmetric and bounded square integrable function. Also, $\int \tilde\kappa(u) du = 1$ and $\int u^2\tilde\kappa(u) du<\infty$.\\
		
		(A14) $h\rightarrow 0$ and $nh\rightarrow\infty$.

\end{assumption}
Then the following holds.
\begin{theorem}\label{th:49}
		Under (A1)-(A9) and (A12)-(A14), for any fixed $ \omega \in [-\pi,\pi]$
		\begin{equation}\label{eq:4226}
				\sqrt{nh}\left(\widehat f_{sp.d}(\omega) - E\widehat{f}_{sp.d}(\omega)\right)\overset{d}{\rightarrow} \mathcal{N}(0, \sigma_\omega^2). 
		\end{equation}
		In the above,  $\sigma_\omega^2 = f_{sp.d}^2(\omega)\int \tilde\kappa^2(u) du$ for $\omega/\pi \notin \mathbb{Z}$ and  $\sigma_\omega^2 = 2f_{sp.d}^2(\omega)\int \tilde\kappa^2(u) du$ for $\omega/\pi \in \mathbb{Z}$.
		Also, for both model-free  and limit model-free bootstrap
		\begin{equation} \label{2ndeq.ofTh:49}
			\sup_{x\in\mathbb{R}} \left\lvert P^* (\sqrt{nh}\left(\widehat f_{sp.d}^*(\omega) - E^*\widehat{f}_{sp.d}^*(\omega)\right)\leq x) - P( \sqrt{nh}\left(\widehat f_{sp.d}(\omega) - E\widehat{f}_{sp.d}(\omega)\right) \leq x)  \right\rvert 
			\overset{P}{\rightarrow} 0.
		\end{equation}
\end{theorem}
\begin{proof}
	See Appendix.
\end{proof}

As is well-known, the bandwidth $h$ governs the trade-off between bias and variance
of $\widehat{f}_{sp.d}(\omega)$. If we choose $h$ in a way that {\it undersmoothing} occurs, 
i.e., when the bias of $\widehat{f}_{sp.d}(\omega)$ is of smaller order of magnitude than its
standard deviation, then equation (\ref{2ndeq.ofTh:49}) holds true with 
$E\widehat{f}_{sp.d}(\omega)$ replaced by ${f}_{sp.d}(\omega)$, i.e., we have
	\begin{equation} \label{3rdeq.ofTh:49}
			\sup_{x\in\mathbb{R}} \left\lvert P^* (\sqrt{nh}\left(\widehat f_{sp.d}^*(\omega) - E^*\widehat{f}_{sp.d}^*(\omega)\right)\leq x) - P( \sqrt{nh}\left(\widehat f_{sp.d}(\omega) -   {f}_{sp.d}(\omega)\right) \leq x)  \right\rvert 
			\overset{P}{\rightarrow} 0.
		\end{equation}
Equation (\ref{3rdeq.ofTh:49}) can then be used to construct model-free
confidence intervals for ${f}_{sp.d}(\omega)$ based on the quantiles of the
bootstrap distribution of $\sqrt{nh}\left(\widehat f_{sp.d}^*(\omega) - E^*\widehat{f}_{sp.d}^*(\omega)\right)$.

\section{Model-free bootstrap   for prediction}
For simplicity, we focus on proving model-free bootstrap validity for one-step ahead prediction; generalizing to $h$- step ahead case is also possible. First of all, we list the following definitions relevant to predictive setting:
 \begin{definition}(Predictive distribution)
	Let $Y_{n+1}$ be the 1- step ahead value of the time series dataset $\{Y_t\}_{t=1}^n$. The conditional distribution $D_{n+1}$ of $Y_{n+1}|
	\underline Y_n$ is called the predictive distribution.
 \end{definition}
\begin{definition}(Predictive root)
	Let $\widehat Y_{n+1}$ be the one step ahead predictor of $Y_{n+1}$
based on the original series $\{Y_t\}_{t=1}^n$, i.e.,  $\widehat Y_{n+1}$
 depends entirely on the data $\{Y_t\}_{t=1}^n$. The predictive root is 
	defined as 
	\begin{equation}
			Y_{n+1} - \widehat{Y}_{n+1}.
	\end{equation}

	We wish to approximate the distribution of the predictive root by the bootstrap procedure.
	Let $L_{\alpha/2}$ and $R_{\alpha/2}$ be the left lower $\alpha/2$ quantile and right upper $\alpha/2$ quantile of the (conditional on $\underline Y_n$)
 distribution of the predictive root. Then, an  exact two-sided $1-\alpha$ prediction interval  for $ Y_{n+1}$ is:
	\begin{equation}
			(\widehat Y_{n+1} + L_{\alpha/2}, \widehat Y_{n+1} + R_{\alpha/2});
	\end{equation}
	in other words, 
	\begin{equation}
			P(\widehat{Y}_{n+1} + L_{\alpha/2} \leq Y_{n+1} \leq  \widehat{Y}_{n+1} + R_{\alpha/2} |\underline Y_n) = 1-\alpha .
	\end{equation}
\end{definition}

\begin{definition}(Bootstrap validity for prediction interval)
	Let $Y_{n+1}^*$ be the one-step ahead value generated through bootstrap procedure. Also 
	let $\widehat{Y}_{n+1}^*$ be the one step ahead predictor of $Y_{n+1}$ conditioning on $\underline Y_n$ with its formula estimated by the bootstrap samples $\{Y_t^*\}_{t=1}^n$. We say that boostrap validity for prediction intervals holds if for any $ \alpha \in (0,1)$:
	\begin{equation} \label{eqofDef5.3}
			L^*_{\alpha/2}\rightarrow L_{\alpha/2}
	\ \ \mbox{and} \ \ 
					R^*_{\alpha/2}\rightarrow R_{\alpha/2}
	\end{equation}
	in probability.
	 Here $L^*_{\alpha/2}$ and $R^*_{\alpha/2}$ denotes the relative $\alpha/2$ quantiles with respect to the conditional 
distribution of the bootstrap predictive root $Y_{n+1}^* - \widehat{Y}_{n+1}^*$.   
\end{definition}
 In other words, the two-sided approximate $1-\alpha$ bootstrap
 prediction interval  for $ Y_{n+1}$ is:
	\begin{equation}
			(\widehat Y_{n+1} + L^*_{\alpha/2}, \widehat Y_{n+1} + R^*_{\alpha/2}),
	\end{equation}
	  while equation (\ref{eqofDef5.3}) would imply
	\begin{equation}
			P(\widehat{Y}_{n+1} + L^*_{\alpha/2} \leq Y_{n+1} \leq  \widehat{Y}_{n+1} + R^*_{\alpha/2} |\underline Y_n) \to 1-\alpha .
	\end{equation}

Given previous definitions, we state the bootstrap validity theorem for prediction intervals.

\begin{theorem}\label{th:51}(Model-free bootstrap validity for 1-step ahead prediction)
	 Assume that the conditional distribution of the predictive root $Y_{n+1} - \widehat{Y}_{n+1}$ is continuous. With assumptions (A1) - (A9),
	  \textbf{Algorithms}  2 and 3  are  asymptotically valid in the sense of Definition 5.3.\\
\end{theorem}
\begin{proof}
See Appendix.
\end{proof}

Unlike the bootstrap for parameter inference, the conditional distribution of the predictive root is not
asymptotically degenerate. Hence, no central limit
theorems are required for parameter estimates since their associated variabilty 
vanishes asymptotically. 

Asymptotic validity of bootstrap prediction interval is of less importance than that of 
confidence intervals.   A good finite-sample prediction interval should incorporate the 
variance of all estimated features, else it will result into {\it  undercoverage}. 
However, as stated above, the property of asymptotic validity has nothing to do with
capturing finite-sample   variability in estimation; see also  Ch. 2.4.1 of Politis (2015). 
 In this respect, the performance of prediction intervals must be quantified in finite-sample
numerical experiments as in the following section.

\section{Numerical results}
We look into  three  data generating models and evaluate the coverage performance of our proposed MFB methods with respect to both confidence intervals and prediction intervals. All three  models satisfy equation \eqref{eq:31} with
\begin{equation}
	f(x) = 
	\begin{cases}
		-\sqrt{-x} & x < 0\\
		\frac{(x+1)^2}{10} & x\geq 0
	\end{cases}
\end{equation}
and $W_t$ generated by an autoregressive   moving average (ARMA) model, namely
$$ W_t-\phi_1 W_{t-1} -\cdots -\phi_p W_{t-p}=
\epsilon_t  + \theta_1 \epsilon_{t-1} +\cdots +\theta_q \epsilon_{t-q}$$
with $\epsilon_t\overset{i.i.d.}{\sim}\mathcal{N}(0,1)$.
The three ARMA models considered are as follows:
in one of the following ways:   
\begin{enumerate}
\item  MA  model of order one, with $\theta_1 = -0.5$ and all other parameters zero.
\item AR model of order one, with $\phi_1 =  0.5$ and all other parameters zero.

\item MA  model of order 30, with $\theta_1 = 2$, $\theta_2 = 1$,   $\theta_k = \frac{10}{k^2}
$ when $ 3\leq k \leq 30$, and all other parameters zero.

\end{enumerate} 

We analyze the performance of model-free and limit model-free bootstrap with both empirical CDF estimator and kernel CDF estimator.

\subsection{Performance for bootstrap confidence intervals}
For confidence intervals, the parameter of interests are: the mean for all three models; lag 1 autocovariance for model 1 and 2; lag 2 autocovariance for model 3. As a comparison, we benchmark our methods with the two most popular alternative methods, i.e. block bootstrap and AR-sieve bootstrap.  

The mean is, of course, a linear statistic. The lag-$k$ autocovariance  can be written 
in the form (\ref{eq:42}) with  
$$ g(Y_t,\cdots,Y_{t+k})=\begin{bmatrix}
		Y_t Y_{t+k}  \\  Y_t
		\end{bmatrix}$$
and $$q(\frac{1}{n-k}\sum_{t=1}^{n-k} g(Y_t,\cdots,Y_{t+k}))
= \frac{1}{n-k}\sum_{t=1}^{n-k} Y_t Y_{t+k}  - [ \frac{1}{n-k}\sum_{t=1}^{n-k} Y_t]^2 .$$
Although we can do block bootstrap on the $Y_t$ and recompute $\widehat\gamma_k$ on the bootstrap data to construct confidence intervals, this procedure will suffer from end effects resulting into an estimator that is biased towards $0$. As a remedy,
let $X_t=(Y_t,\cdots,Y_{t+k})$ as we did in Section \ref{section:vsfls},
then a block bootstrap on the $X_t$ data will relieve this problem and it is equivalent to the so-called blocks-of-blocks bootstrap
on the $Y_t$ data; see \cite{politis1992}, as well as 
Paradigm 12.8.11 of the book by \cite{McElroyPolitis2019}. 
For our purposes, we used first level of blocking with $k=5$ that can be used to capture
autocovariances up to lag 4.

For the block bootstrap on the $X_t$ data, we can choose the   block size $b$  as $b = const*n^{1/3}$ where {\it const} is selected according to \cite{blockchoice}.
 As for the AR-sieve bootstrap, the order $p$ of the fitted AR model is selected through minimizing the AIC; see \cite{buhlmann1997}.

The metric we use for comparison is the empirical coverage rate for the bootstrap intervals.
To elaborate, let $N = 1000$ be the number of replicated experiments, $n \in \{100,200,500,1000\}$ be the length of samples, $\alpha = 0.05$ and $B = 250$ the number of bootstrap replications. So 
for    experiment $i\in \{1,\ldots, N\}$, we generate a time series sample  of length $n$, and use different methods to construct a $1-\alpha$ bootstrap confidence interval $(L_i,R_i)$ based on the $B$   bootstrap replicates. Let $\theta_0$ denote the true parameter of interest; then the empirical coverage rate for the whole experiment is:
\begin{equation}
CVR_{\theta_0} = \frac{1}{1000} \sum_{i=1}^{1000} \mathit{I}_{\theta_0\in(L_i,R_i)}
\end{equation}  
The purpose of looking at CVR is two fold: first, comparing CVR values in a finite-sample size setting will tell which method has better performance. Second, as the sample size $n$ gets large, asymptotic validity can be observed by checking whether the empirical CVR converges to the  nominal $100(1-\alpha)\%$ percent, which will provide numerical grounds to our previous conclusions.

The following tables contains CVR results for different settings all 
with nominal $\alpha = 0.05$ (95\% confidence interval).
The acronyms in the ``Method" column of Table 1 have the following meaning. "MF" for model-free bootstrap; "LMF" for limit model-free bootstrap; "emp" for empirical CDF estimator used in the bootstrap procedure; "ker" for kernel CDF estimator; "BB" for   block bootstrap;
 "AR-sieve" for the autoregressive sieve bootstrap.

For the mean parameter (see Table 1), all the methods being compared are theoretically valid for all  three models. We can observe  that for each method, as $n$ increases the coverage rate approaches 95\% albeit the speed of convergence can be different. The AR sieve holds an obvious advantage against other methods for model 1 with coverage close to 95\% even at $n = 100$, whereas the proposed model-free bootstrap works on par with the block bootstrap. However for the other two models, MF and LMF bootstrap with kernel CDF estimator has a noticeable advantage over AR-sieve and block bootstrap, especially for model 3 which has a more complex data generating process.

 Interestingly,  the behavior of model-free bootstrap using the empirical CDF  is almost the same with block bootstrap. 
However the advantage goes away for large $n$ which is fortunate since  calculating the
quantile inverse of
a kernel CDF (which is needed in the bootstrap algorithm)is  computationally expensive for large sample size. 

\begin{centering}
	\begin{table}[]
		\centering
		\begin{tabular}{p{2cm}|p{1cm}|p{1cm}|p{1cm}|p{1cm}|p{1cm}}
			
			\toprule
			\hline
			Method & n=100 & 200 & 500 & 1000 &2000 \\
			\hline
			\multicolumn{6}{c}{model 1}\\
			\hline
			MF-ker &92.1 & 92.5 &93.2 &93.6 & 94.8\\
			\hline
			LMF-ker &92.3& 92.5 & 93.0 & 94.7 &93.9  \\
			\hline
			MF-emp & 91.6 & 92.4 & 93.1 &94.0  &94.4  \\
			\hline
			LMF-emp &92.0  &92.9  &93.5  &93.8  &94.0  \\
			\hline
			BB  &90.4 &93.1 &94.1 &95.0 &95.0 \\
			\hline
			AR-sieve &94.0 &93.6 &94.0 &95.2 &95.6 \\
			\hline
			\multicolumn{6}{c}{model 2}\\
			\hline
			MF-ker&91.4& 92.4& 93.2 & 93.2 & 93.7\\
			\hline
			LMF-ker &88.5 &92.1 &93.3 &93.3 & 94.4\\
			\hline
			MF-emp & 88.7& 90.9& 92.8& 94.1& 92.9\\
			\hline
			LMF-emp &85.3 & 88.4& 92.4& 93.0& 92.9 \\
			\hline
			BB  &85.2 &89.6 &91.3 &92.3 &93.0 \\
			\hline
			AR-sieve &90.6 &92.9 &92.8 & 93.6 &94.2 \\
			\hline
			\multicolumn{6}{c}{model 3}\\
			\hline
			MF-ker &87.9 &89.3 &90.3 &91.7 &93.3 \\
			\hline
			LMF-ker & 86.9& 88.1& 90.5& 92.1&93.1 \\
			\hline
			MF-emp &82.7 &85.5 & 89.6& 91.4& 92.7\\
			\hline
			LMF-emp &80.9 &85.8 &89.5 &91.4 &92.8 \\
			\hline
			BB  &82.2 &85.4 &87.3 &89.8 &91.4 \\
			\hline
			AR-sieve &83.6 &88.1 & 90.0& 91.9& 92.2\\
			\hline
			\bottomrule
			
		\end{tabular}
		\caption{Empirical CVR for   the mean parameter across 3 models}
		\label{tab:my_label}
	\end{table}
\end{centering}

Tables 2 and 3 present the empirical CVR for the autocovariance  
with different models. Both lag 1 and lag 2 autocovariances 
were considered; for conciseness, we present lag 1  results from models 1 and 2
(see Table 2), and  lag 2  results from model 3 (see Table 3).
Recall that   
the AR-sieve bootstrap    is not asymptotically valid in general
for the    autocovariance. Comparing Tables 2 and 3, we see that 
the AR-sieve bootstrap works well for the autocovariance in models 1 and 2
but not with data from model 3; see
 Table 3 where the  AR-sieve CVR appears to converge  to around 86\% instead of  the
95\% nominal level.  

We can also observe asymptotic validity of model-free bootstrap methods 
manifesting from Table 2 and 3. 
Furthermore,  model-free methods 
appear to enjoy a faster convergence towards nominal compared to block bootstrap. 
Interestingly, limit model-free with kernel CDF estimator works significantly better than  other
methods  across all 3 models.

\begin{centering}
\begin{table}[]
	\centering
	\begin{tabular}{p{2cm}|p{1cm}|p{1cm}|p{1cm}|p{1cm}|p{1cm}}
		\toprule
		\hline
		Method & n=100 & 200 & 500 & 1000 &2000 \\
		\hline
		\multicolumn{6}{c}{model 1}\\
		\hline
		MF-ker & 82.0& 88.2 & 90.5& 92.2& 91.1\\
		\hline
		LMF-ker & 89.6 &94.3 & 94.0  & 94.8 &94.1  \\
		\hline
		MF-emp &  80.0& 88.3 & 90.1 & 91.6 &  92.0\\
		\hline
		LMF-emp & 89.2 & 93.5 &94.0 & 95.0 & 93.3 \\
		\hline
		BB  &86.0 &88.8 & 91.5&91.1 & 93.0\\
		\hline
		AR-sieve &94.3 & 94.8& 94.6& 96.3& 95.0\\
		\hline
		\multicolumn{6}{c}{model 2}\\
		\hline
		MF-ker & 75.0&  86.5& 94.2& 94.6& 94.6\\
		\hline
		LMF-ker & 92.2 & 94.2 & 95.4  & 94.0 & 93.4 \\
		\hline
		MF-emp & 68.6 & 82.9 & 90.2 & 91.9 &94.2  \\
		\hline
		LMF-emp & 76.9 & 82.5 &  89.0& 91.6 & 94.4 \\
		\hline
		BB  & 82.3& 86.2& 89.0 & 90.5&91.5 \\
		\hline
		AR-sieve &88.3 &91.8 &95.1 &95.5 &95.5 \\
		\hline
		\bottomrule
		
	\end{tabular}
	\caption{Empirical CVR for lag 1 autocovariance for the first 2 models}
	\label{tab:my_label}
\end{table}
\end{centering}
 
 \begin{centering}

\begin{table}[h]
	\centering
	\begin{tabular}{p{2cm}|p{1cm}|p{1cm}|p{1cm}|p{1cm}|p{1cm}}
		\toprule
		Method & n=100 & 200 & 500 & 1000 &2000 \\
		\hline
		\hline
		MF-ker & 76.9& 87.8& 91.7& 93.0& 94.1\\
		\hline
		LMF-ker &77.7 &86.0 &91.3 &92.4 &93.8 \\
		\hline
		MF-emp &68.2 & 80.9& 86.8& 91.6& 93.5\\
		\hline
		LMF-emp &65.1 &77.5 &85.8 &89.5 &93.5 \\
		\hline
		BB  &64.0 &73.9 &83.2 & 87.3& 89.4\\
		\hline
		AR-sieve &69.4 &76.4 & 84.9& 86.1& 86.1\\
		\bottomrule
		
	\end{tabular}
	\caption{Empirical CVR for lag 2 autocovariance for model 3}
	\label{tab:my_label}
\end{table}
 \end{centering}

\subsection{Performance of bootstrap prediction intervals}
We now move on to the prediction performance for the three models. The empirical coverage rate is defined similarly as
\begin{equation}
	CVR_{Y_{n+1}} = \sum_{i=1}^{1000} I_{Y^{(i)}_{n+1}\in(\widehat Y_n^{(i)} + L_i, \widehat Y_n^{(i)}  + R_i)}
\end{equation}
where $(L_i, R_i)$ are sample quantiles of bootstrap predictive root generated by bootstrap and $Y_{n+1}^{(i)}$ is the $n+1$ value of the time series sample from the $i^{th}$ experiment.

Since block bootstrap is not a viable method for generating 1-step ahead prediction value, we benchmark predictive performance
of the model-free methods comparing them   with the AR-sieve bootstrap.
The bootstrap samples are generated through a forward bootstrap manner as described in Algorithm 3.1 of \cite{PanPolitis2016}; see also  \cite{ALONSO20021}.
Table 4 provides the empirical CVR with $n \in \{100,200,300,500\}$; the samples sizes
are smaller  compared to our previous simulations because of the increasing computational cost for prediction. However, it is reassuring that all methods considered,
i.e., the 4 model-free variations as well as the AR-sieve bootstrap,
produce prediction intervals with CVR close to the nominal 95\% even with
$n$ as low as 200. It is difficult to do a finer  comparison of these 5 bootstrap methods 
since, for computational reasons, we had to choose a small number of bootstrap replications
($B=250$). Ongoing work includes devising an analog of the `Warp-Speed' method of
\cite{giacomini_politis_white_2013} that will speed up Monte Carlo experiments involving bootstrap 
in the case of prediction intervals.

\begin{centering}
\begin{table}[]
	\centering
	\begin{tabular}{p{2cm}|p{1cm}|p{1cm}|p{1cm}|p{1cm}}
		\toprule
		\hline
		Method & n=100 & 200 & 300 & 500  \\
		\hline
				\multicolumn{5}{c}{model 1}\\
		\hline
		MF-ker &   95.0 &  94.6 & 91.6  &91.8 \\
		\hline
		LMF-ker & 94.1   & 96.6  &  93.9 & 93.0\\
		\hline
		MF-emp &94.0  & 94.2  & 92.2  &94.2    \\
		\hline
		LMF-emp &92.2  &  93.9 &  94.0 & 93.8   \\
		\hline
		AR-sieve &  94.4  & 94.6  & 94.0 & 93.2 \\
		\hline
		\multicolumn{5}{c}{model 2}\\
		\hline
		MF-ker & 93.6   & 94.0  &  92.2 &93.2 \\
		\hline
		LMF-ker & 90.2   & 95.3  & 96.4  & 92.8\\
		\hline
		MF-emp &  94.1 & 93.4  &  95.0 & 94.2   \\
		\hline
		LMF-emp &  89.2  &90.7   & 95.0  & 93.6    \\
		\hline
		AR-sieve &  94.0  & 93.2  & 94.0 & 94.4  \\
		\hline
		\multicolumn{5}{c}{model 3}\\
		\hline
		MF-ker &93.8& 94.4&95.4&94.8 \\
		\hline
		LMF-ker &91.8 &96.7 &95.6&  93.6\\
		\hline
		MF-emp & 92.2&96.2& 96.4& 95.2\\
		\hline
		LMF-emp &92.8 &94.8 &96.4&93.2 \\
		\hline
		AR-sieve & 91.4& 94.8&93.2& 93.0\\
		\hline
		\bottomrule
\end{tabular}
\caption{Empirical CVR for prediction intervals across the three models}

\end{table}

\end{centering}

\vskip .1in
{\bf Acknowledgments.} Many thanks are due to 
Jens-Peter Kreiss, Stathis Paparoditis, and the participants of the August 2015 Workshop
on {\it Recent Developments in Statistics for Complex Dependent Data}, Loccum (Germany),
where the seeds of this work were first presented. This research was partially supported
by NSF grants DMS 16-13026 and  DMS 19-14556.


\bibliography{bbtex}

\section{Appendix}

\begin{proof}[Proof of Lemma \ref{lm:33}]
	1. Let $\gamma_m = m^{-C/\theta}$ and $\delta_m = m^{-C}$. By Markov's inequality, 
	\begin{equation}
	\begin{split}
	P(|Y_t - Y_t^{(m)}|>m^{-C/\theta}) &= P(|Y_t - Y_t^{(m)}|^p>m^{-Cp/\theta}) \\&\leq \frac{E(|Y_t - Y_t^{(m)}|^p)}{m^{-Cp/\theta}} \\& \ll \frac{m^{-pA}}{m^{-Cp/\theta}} = m^{-C} \\
	\end{split}
	\end{equation}
	
	2. $$\sum_{m=0}^\infty\norm{Y_t - Y_t^{(m)}}_p\leq\sum_{m=0}^\infty\delta(m).$$ Since $\delta(m)\ll m^{-A} $ with $A>1$, $\delta(m)$ is summabe. Thus (C3) holds.\\
	
	3. Let $\epsilon_0'$ be the new independent sample to be used in $\delta_p(t)$. Let $\epsilon_0'',\epsilon_{-1}'',\cdots$ be an infinite sequence of i.i.d. samples from $F_\epsilon$. By triangle inequality we have:
	\begin{equation}
	\begin{split}
	\delta_p(m) & = \norm{h(\cdots,\epsilon_{-1},\epsilon_0,\cdots,\epsilon_{m}) - h(\cdots,\epsilon_{-1},\epsilon_0',\cdots,\epsilon_{m}) }_p\\
	& \leq   \norm{h(\cdots,\epsilon_{-1},\epsilon_0,\cdots,\epsilon_{m}) - h(\cdots,\epsilon_{-1}'',\epsilon_0'',\epsilon_1,\cdots,\epsilon_{m}) }_p\\ &+ \norm{h(\cdots,\epsilon_{-1}'',\epsilon_0'',\epsilon_1,\cdots,\epsilon_{m}) - h(\cdots,\epsilon_{-1},\epsilon_0',\cdots,\epsilon_{m})}_p\\
	& = 2 \norm{Y_m - Y_m^{(m)}}_p
	\end{split}
	\end{equation}
	Thus (C3) implies (C0).	\\
	
	4. See page 2443 in \cite{berkes2011}.
\end{proof}

\quad 

\begin{proof}[Proof of Lemma \ref{th:42}]
		Equation (\ref{eq:49}) is deduced by result in Lemma \ref{lm:33}. Then
		by Theorem \ref{th:35} 
		\begin{equation}
		\begin{split}
		\sup_{t\in[n]}|\widehat{U}_t - U_t| &\leq \frac{1}{n}\sup_{s\in\mathbb{R}}|\sum_{k=1}^n(I(Y_k\leq s) - F_Y(s))|\\
		& = \frac{1}{n} \sup_{s}|R(s,n)|\\
		& \leq \frac{1}{\sqrt{n}}\sup_s|K(s,1)| + o(n^{-1/2} (\log n)^{-\alpha}), a.s.
		\end{split}
		\end{equation}
		where $K(s,1)$ is a centered Gaussian process with covariances
		\begin{equation}\label{eq:412}
		\Gamma(s,s') = \sum_{k\in\mathbb{Z}}E(I(Y_0\leq s) - F_Y(s))(I(Y_k\leq s') - F_Y(s')) 
		\end{equation}
		absolutely convergent $\forall s,s'\in\mathbb{R}$. 
		We also need the limiting Gaussian process $K(s,1)$ to be tight such that $\sup_{s} |K(s,1)| = O_p(1)$, which usually is true but not mentioned in Theorem \ref{th:35}. Another way to show this is to prove certain continuity and boundedness condition for the covariance $\Gamma(s,s')$.

	Note that $K(s,1)$ can be reparametrized as $K'(u), u = F_Y(s)$ such that the centered gaussian process is now living in a bounded domain $T = [0,1]$(reparametrization of the empirical process). Also, since $F_Y$ is absolutely continuous and strictly increasing, $\sup_{s\in\mathbb{R}} |K(s,1) | = \sup_{t\in[0,1]} |K'(t)|$. Let $\Gamma'(t,t')$ denote the covariance kernel of $K'(u)$. Obviously $\Gamma'(t,t') = \Gamma(F_Y^{(-1)}(t),F_Y^{(-1)}(t'))$, for which we have $\Gamma'(0,0) = \Gamma'(1,1) = 0$. Thus the Gaussian process $K'(u)$ is pinned to $0$ a.s. at enpoints $0$ and $1$. \\
	We claim that 
	\begin{itemize}
		\item (1) $\Gamma'(t,t')$ is uniformly bounded for $t,t'\in [0,1]^2$.
		\item (2) $\Gamma'(t,t')$ is continuous in $t$ and $t'$.
		\item (3) $K'$ satisfies the Kolmogorov-Chentsov condition: $\exists \alpha,\beta,C>0$, such that
	$$	E|K'(t) - K'(s)|^\alpha\leq C|t-s|^{1+\beta}$$
	\end{itemize}
	(2) is a consequence of (1). To see this, note that 
	$\forall t$,$\forall t_1<t_2$, $|\Gamma'(t_1,t) - \Gamma'(t_2,t)| = |\sum_{k\in \mathbb{Z}}E(I(U_0\leq t_1) - t_1)(I(U_k\leq t) - t) - E(I(U_0\leq t_2) - t_2)(I(U_k \leq t) - t)| = E\left[I(t_1\leq U_0\leq t_2) - (t_2-t_1)\right]\left[\sum_{k\in \mathbb{Z}}(I(U_k \leq t) - t)\right]$. By previous theorem, the expectation exists $\forall t$ and $t_1<t_2$, and uniformly bounded by $2\sup_{t\in[0,1]} |\Gamma'(t,t)|.$ As $|t_2-t_1|\rightarrow 0$, the random variable inside the expectation converges to $0$ a.s. Thus by dominated convergence theorem we have $|\Gamma'(t_1,t) - \Gamma'(t_2,t)| \rightarrow 0 $.\\
	
	For (1), by positive definiteness of covariance kernel we have that
	\begin{equation}
	\begin{split}
	\sup_{(t,t')\in [0,1]^2}|\Gamma'(t,t')| &=  \sup_{t\in[0,1]}|\Gamma'(t,t)| = \sup_{t\in [0,1]}|\sum_{k\in\mathbb{Z}} E(I(U_0\leq t) - t)(I(U_k\leq t) - t)|\\
	&= \sup_{t\in[0,1]} |\sum_k E\left[I(U_0\leq t)I(U_k\leq t) -t^2\right]|\\
	& = \sup_{t\in[0,1]} |\sum_k \int_{(-\infty,g^{-1}(t)]^2} f_{W}^{(k)}(x,y) - f_W(x)f_W(y) dxdy|\\
	&\leq \int_{\mathbb{R}^2} \sum_k |f_{W}^{(k)}(x,y)-f_W(x)f_W(y)| dxdy \quad\quad\cdots (*)
	\end{split}
	\end{equation}
	Where $g = F_Y\circ f$ is a monotone (increasing) function; $f_{W}^{(k)}$ is the bivariate normal density of $(W_0,W_k)$; $f_W$ is the normal density of $W_t$. We show that $(*)$ is bounded.\\
		Rewrite $f_{W}^{(k)}(x,y)$ as $f_{xy}(a_k,b_k) =  \frac{1}{2\pi}a_k^{-1/2}\exp{\{-\frac{1}{2}a_k^{-1}b_k\}}$,with $a_k = \sigma^2_W(0) - \sigma_W^2(k)$ and $b_k = \sigma_W(0)(x^2 + y^2) - 2\sigma_W(k)xy$, where $\sigma_W(k)$ is the lag-$k$ autocovaraince for $W_t$.  \\
		\begin{equation}
		\nabla f  = 
		\begin{bmatrix}
		\frac{\partial f}{\partial a_k} \\
		\frac{\partial f}{\partial b_k}
		\end{bmatrix}
		 = \begin{bmatrix}
		\frac{1}{2} a_k^{-\frac{3}{2}}(a_k^{-1}b_k -1)e^{-\frac{1}{2}a_k^{-1}b_k}\\
		-\frac{1}{2}a_k^{-\frac{3}{2}}e^{-\frac{1}{2}a_k^{-1}b_k}
		\end{bmatrix}
		\end{equation}
		Then 
		\begin{equation}
		\begin{split}
			& \int_{\mathbb{R}^2} |f_{xy}(a_k,b_k) - f_{xy}(\sigma^2_W(0),\sigma_W(0)(x^2 + y^2))| dxdy \\
			& = 
			\int_{\mathbb{R}^2}|\nabla f^T(\sigma^2_W(0),\sigma_W(0)(x^2 + y^2))\begin{bmatrix}-\sigma^2_W(k)\\-2\sigma_W(k)xy \end{bmatrix}| dxdy + o(\sigma_W(k) + \sigma_W^2(k))
		\end{split}	
		\end{equation}
		and
		\begin{equation}
		\begin{split}
		(*) &= \int_{\mathbb{R}^2}\sum_k |f_{xy}(a_k,b_k) - f_{xy}(\sigma^2_W(0),\sigma_W(0)(x^2 + y^2))| dxdy\\
		& = O\left( \int_{\mathbb{R}^2}\sum_k|\nabla f^T(\sigma^2_W(0),\sigma_W(0)(x^2 + y^2))\begin{bmatrix}-\sigma^2_W(k)\\-2\sigma_W(k)xy \end{bmatrix}| dxdy\right)\\
		& = O\left(\sum_k \sigma_W(k) + \sum_k \sigma_W^2(k)\right) 
		\end{split}
		\end{equation}

		With $\{W_t\}$ satisfying (C2), $\sum_k \sigma_W(k)$ is absolutely convergent, then so is $\sum_k \sigma_W^2(k)$, the above quantity is bounded. \\
		For (3), Since $K'$ is a centered gaussian process, We have for $s<t$
		\begin{equation}
E|K'(t)-K'(s)|^\alpha = \sqrt{\frac{2}{\pi}}\int_{x\in(0,\infty)} x^\alpha \frac{1}{\sigma_{s,t}}e^{-\frac{x^2}{2\sigma_{s,t}^2}} dx = C_3 \sigma_{s,t}^{\alpha-1}	
		\end{equation}

		Where $\sigma_{s,t}^2 = \sum_{k\in\mathbb{Z}}E(I(U_0\in(s,t]) - (t-s))(I(U_k\in(s,t]) - (t-s))\quad\cdots(**)$. It suffice to show that $\exists \alpha>1,\beta, C>0$ s.t. $\sigma_{s,t}^2 \leq C(t-s)^{\frac{2(1+\beta)}{\alpha -1 }}.$ Note that
		\begin{equation}\label{eq:79}
			(**) \leq \int_{(x,y)\in A_{s,t}^2} \sum_k|f_W^{(k)}(x,y) - f_W(x)f_W(y)| dxdy
		\end{equation}
		
		Where $A_{s,t}\subset \mathbb{R}$ is such that $\int_{A_{s,t}} f_W(x) dx = t-s$. With same Taylor expansion arguments, there exists positive constants $C_i$ such that the right hand side of equation (\ref{eq:79}) is bounded by $ C_1E(W^2 I(W\in A_{s,t})) + C_2 E(|W|I(W\in A_{s,t})) + C_3E(I(W\in A_{s,t}))$, where
		$W$ is a normal random variable with density $f_W$. 
		Since in the calculation of the expectations the dominating term is the exponential decay, $\forall \epsilon >0$, 
		$E(W^2 I(W\in A_{s,t})) = o((t-s)^{1-\epsilon})$, and so is $E(|W|I(W\in A_{s,t}))$. Thus $\exists \epsilon_0 \in (0,1)$,
		$\left(**\right) \leq C(t-s)^{\epsilon_0}$. The Kolmogorov continuity condition is satisfied. By Borell's inequality(Chapter 2.1, \cite{adler}), the tail probability of the supremum is bounded by the tail probability of the Gaussian distribution with variance $\sup_{t\in [0,1]} \Gamma'(t,t) <\infty$. Thus we have proved the desired result.\\
\end{proof}

\quad

\begin{proof}[Proof of Lemma \ref{th:43}]
	By calculations in Theorem 1.2 in \cite{li_raccine},
	\begin{equation}
	E(\widehat{F}_h(x)) - F(x) = \frac{c_2}{2}h^2F^{(2)}(x) + o(h^2).
	\end{equation}
	where $c_2 = \int x^2 k(x) dx<\infty$. Thus with $F^{(2)}(x)$ bounded and $h = o(n^{-1/4})$,
	\begin{equation}\label{eq:444}
	\sup_{x\in \mathbb{R}} |E(\widehat{F}_h(x)) - F(x)| = o(n^{-1/2}).
	\end{equation}
	Also by (A7) we have 
	\begin{equation}
	\sqrt{n}(\widehat{F}_h(x) - E(\widehat{F}_h(x)))  = G_h(x) + o_p(1)
	\end{equation}
	and 
	$G_h(x)$ is a tight centered Gaussian process with covariance 
	$\Gamma_{G_h}(x,x') = \sum_{k\in \mathbb{Z}}E\left(K(\frac{x-Y_0}{h}) - E(K(\frac{x-Y_0}{h}))\right)\left(K(\frac{x - Y_k}{h}) - E(K(\frac{x-Y_k}{h}))\right)$$\rightarrow \Gamma(x,x')$ as $h(n)\rightarrow 0$, $\forall x,x'\in\mathbb{R}$, where $\Gamma$ is given in equation (\ref{eq:36}). As shown in Lemma \ref{th:42}, the limit Gaussian process is tight.
	This is sufficient for
	\begin{equation}
	\sup_{x\in\mathbb{R}} \left\lvert\widehat{F}_h(x) - E(\widehat{F}_h(x))\right\rvert = O_p(\frac{1}{\sqrt{n}}).
	\end{equation}
	Together with equation (\ref{eq:444}) we have $\sup_{x\in\mathbb{R}} |\widehat{F}_h(x) - F(x)| = O_p(\frac{1}{\sqrt{n}})$. Therefore $\sup_{t\in[n]}|\widehat{U}_t^{(h)} - U_t| = O_p(\frac{1}{\sqrt{n}})$.
	
\end{proof}

\quad

\begin{proof}[Proof of Theorem \ref{th:44}]
	Note that 
	\begin{equation}
	\sum_{k=0}^{\lfloor c_{\kappa}l\rfloor}|\wwidehat{\sigma}(k) - \sigma(k)|\leq 
	\underbrace{\sum_{k=0}^{\lfloor c_{\kappa}l\rfloor}|\wwidehat{\sigma}(k) - \widehat\sigma(k)|}_{(1)} + \underbrace{\sum_{k=0}^{\lfloor c_{\kappa}l\rfloor}|\widehat{\sigma}(k) - \sigma(k)|}_{(2)}
	\end{equation}
	First of all, for (2), by Lemma \ref{lm:34}, $\{Z_t\}$ is (C2) with $A>1$ and therefore satisfies (C0). Then by Theorem\ref{th:37}, with $l = o(n^{\frac{p-2}{p}})$,  $(2) \overset{P}{\rightarrow} 0$.\\
	For (1), 
	let $\tilde{Z}_t = \tilde{\Phi}^{-1}(U_t)$ and $\tilde{\widehat{Z}}_t =  \tilde{\Phi}^{-1}(\widehat{U}_t)$, where $\tilde{\Phi}$ is defined in Section 2. Then $\tilde{\Phi}^{-1}$ is a  function bounded by $c$ and $-c$.  \\
	 Now consider:
	\begin{equation}
	\begin{split}
	\sum_{k=0}^{\lfloor c_{\kappa}l\rfloor}|\wwidehat{\sigma}_Z(k) - \widehat{\sigma}_Z(k)| &= \sum_{k=0}^{\lfloor c_\kappa l \rfloor} \frac{1}{n}\lvert\sum_{t = 1}^{n-k} \left(Z_tZ_{t+k} - \tildehat{Z_t}\tildehat{Z}_{t+k}\right)\rvert\\
	& = \sum_{k=0}^{\lfloor c_\kappa l \rfloor} \frac{1}{n}|\sum_{t = 1}^{n-k} \left(Z_tZ_{t+k} - \Tilde{Z}_t\Tilde{Z}_{t+k} + \Tilde{Z}_t\Tilde{Z}_{t+k}  - \tildehat{Z_t}\tildehat{Z}_{t+k}\right)|\\
	&\leq \sum_{k=0}^{\lfloor c_\kappa l \rfloor}|\frac{1}{n}\sum_{t = 1}^{n-k} \left(Z_tZ_{t+k} - \Tilde{Z}_t\Tilde{Z}_{t+k}\right)|\quad\left(T_1\right) \\&+ \sum_{k=0}^{\lfloor c_\kappa l \rfloor}|\frac{1}{n} \sum_{t = 1}^{n-k} \left(\Tilde{Z}_t\Tilde{Z}_{t+k}  - \tildehat{Z_t}\tildehat{Z}_{t+k}\right) |\quad\left(T_2\right)
	\end{split}
	\end{equation}
	Let $X\lesssim Y$ denote $\exists c>0, X\leq cY$.
	
	For $T_2$,
\begin{equation}
\begin{split}
T_2& = \sum_{k=0}^{\lfloor c_\kappa l \rfloor} |\frac{1}{n} \sum_{t = 1}^{n-k} \Tilde{Z}_t\Tilde{Z}_{t+k} - \Tilde{Z}_t\tildehat{Z}_{t+k} + \Tilde{Z}_t\tildehat{Z}_{t+k} - \tildehat{Z_t}\tildehat{Z}_{t+k}  |\\
& = \sum_{k=0}^{\lfloor c_\kappa l \rfloor}|\frac{1}{n}\sum_{t = 1}^{n-k} \Tilde{Z}_t\left(\Tilde{Z}_{t+k} - \tildehat{Z}_{t+k} \right) + \tildehat{Z}_{t+k}\left(\Tilde{Z}_t - \tildehat{Z_t}\right) |\\
& \lesssim \sum_{k=0}^{\lfloor c_\kappa l \rfloor} \frac{c}{\phi(c)}\sup_{t\in [n]} |U_t - \widehat{U}_t|\\
& = O_p( lce^{\frac{c^2}{2}} n^{-1/2}). 
\end{split}
\end{equation}

Choose $c(n)\rightarrow\infty$ such that $lce^{\frac{c^2}{2}} = o(\sqrt{n})$ makes $T_2\rightarrow 0$ in probability.

Meanwhile,
\begin{equation}
\begin{split}
ET_1 & \leq \sum_{k=0}^{\lfloor c_\kappa l \rfloor} \frac{1}{n}E| \sum_{t = 1}^{n-k} Z_tZ_{t+k} - \Tilde{Z}_t\Tilde{Z}_{t+k}| \\
& = \sum_{k=0}^{\lfloor c_\kappa l \rfloor} \frac{n-k}{n} E|Z_tZ_{t+k} - Z_t\tilde{Z}_{t+k} +Z_t\tilde{Z}_{t+k} -  \Tilde{Z}_t\Tilde{Z}_{t+k}|\\
&\lesssim l E|Z_t(Z_{t+k} - \tilde{Z}_{t+k})|\\
&\leq l \norm{Z_t}_2\norm{Z_{t+k}-\tilde{Z}_{t+k}}_2\\
&= O(l(ce^{-c^2/2})^{1/2})
\end{split}
\end{equation}
with specific calculations for $E(Z_{t+k}-\tilde{Z}_{t+k})^2$ using $Z_{t+k}\sim \mathcal{N}(0,1)$.
Thus by choosing $l,c \rightarrow\infty$ such that $lce^{c^2/2} = o(\sqrt{n})$ and $lc^{1/2}e^{-c^2/4} = o(1)$ we have both $T_1\overset{P}{\rightarrow} 0$ and $T_2\overset{P}{\rightarrow} 0$.

Equation (\ref{eq:414}) is a direct consequence of equation (\ref{eq:413}). Note that both $\wwidehat{\Sigma}_n$ and $\Sigma_n$ are Toeplitz, then:

\begin{equation}
	\begin{split}
	\norm{\wwidehat{\Sigma}_n - \Sigma_n}_{op}& \leq 	\sqrt{\norm{\wwidehat{\Sigma}_n - \Sigma_n}_{1}	\norm{\wwidehat{\Sigma}_n - \Sigma_n}_{\infty}} =
	\sum_{k\in \mathbb{Z}} |\wwidehat{\sigma}(k) - \sigma(k)| \\
	& = \sum_{|k|\leq \lfloor c_\kappa l\rfloor} |\wwidehat{\sigma}(k) - \sigma(k)| + \sum_{|k|>\lfloor c_\kappa l\rfloor}| \sigma(k)|
	\end{split}
\end{equation}
First term is shown to converge to 0 in probability. For second term, 
note that $\sigma(k)$ is absolutely summable by (C2) condition on $Z_t$, then with $l(n)\rightarrow\infty$, second term converge to 0. Since the spectral density of $W_t$ is both bounded and bounded away from 0, with $Z_t$ a linear transform from $W_t$, same holds for $Z_t$. Then by Theorem 2 of \cite{mcmurrypolitis2010}, $\norm{\wwidehat{\Sigma}_n^{-1} - \Sigma_n^{-1}}_{op}\overset{P}{\rightarrow} 0$.
\end{proof}

\quad

\begin{proof}[Proof of Lemma \ref{lm:45}]
Before proving the assertion, some preliminary lemmas are required and listed below:
\begin{lemma}\label{lm:71}
	Given that $\widehat{F}$ is uniformly consistent for $F_Y$ and $F_Y$ is continous and strictly increasing, then with,$\forall p \in (0,1)$, $\widehat{F}^{-1}(p) \overset{P}{\rightarrow} F_Y^{-1}(p)$.
\end{lemma}
\begin{proof}
	See  Lemma 1.2.1 in \cite{Politis1999}
\end{proof}
\begin{lemma}\label{lm:72}
	Let $U\sim Unif(0,1)$. Then $\widehat Y = \widehat F^{-1}(U) \sim \widehat F$ and $Y = F_Y^{-1}(U)\sim F_Y$, With $d_\infty(\widehat F, F_Y) = \sup_x |\widehat{F}(x) - F_Y(x)| \overset{P}{\rightarrow} 0$, $\widehat Y \overset{d}{\rightarrow} Y$ in probability.
\end{lemma}
\begin{proof}
	For the first part, see \cite{PITresult}. The second part is straightforward.
\end{proof}

	By Lemma \ref{lm:72} and Theorem \ref{th:44}, also by continuous mapping theorem, (1) in Lemma \ref{lm:45} is straightforward. We focus on proving (2).
	
	Note that for both ecdf and kernel cdf with appropriate bandwidth choice, $\sup_x |\widehat{F}(x) - F_Y(x)| \overset{P}{\rightarrow} 0$ with relative assumptions. Then by continuous mapping theorem, $\forall k \in\mathbb{N}$,
	$\tilde\Phi^{-1}(\widehat{F}_Y(Y_k))$ converges in probability to 
	$Z_k = \Phi^{-1}(F(Y_k))$.  Specifically,  $\forall d\leq n$, 
	$\tildehat{\underline Z}_d \overset{P}{\rightarrow} \underline Z_d$. Then with assumption (A9)
	\begin{equation}
	\widehat{\underline\xi}_d\overset{P}{\rightarrow} \underline \xi_d\sim \mathcal{N}(0,I_d)
	\end{equation}
	
 Thus  by Lemma 3.4 of \cite{Kallenberg1997}, sequence $\widehat \xi = (\widehat{ \xi}_1,\cdots)$ converges in probability to $\xi = (\xi_1,\cdots)$ with repect to the metric $\rho(\widehat\xi,\xi) = \sum_{k = 1}^\infty 2^{-k}|\widehat\xi_k - \xi_k|$, where  $\xi_t\overset{i.i.d.}{\sim} \mathcal{N}(0,1)$. Let $\bar F_{\widehat\xi}$, $\bar F_{\xi}$ be the empirical CDF of $\widehat \xi$ and $\xi$ respectively, i.e. $\bar F_{\xi} (x) = \frac{\sum_{k=1}^n I(\xi_k\leq x)}{n}$ and $L_{n,x}(\xi) =\bar F_{\xi}(x)$, $L_{x}(\xi) =\lim_{n\rightarrow\infty}L_{n,x}(\xi)$.  By using continous mapping theorem(Theorem 18.11, \cite{vaart_1998}) on $L_{n,x}(\cdot)$ and $L_{x}(\cdot)$, we have for $x$ a.e., $\bar F_{\widehat\xi}(x) - \bar F_{\xi}(x)\overset{P}{\rightarrow} 0$. Since $\sup_{x}|{\bar F_{\xi}(x) - \Phi(x)|}\rightarrow 0, a.s.$, along with continuity of $\Phi$ and by P\'olya's theorem,  $\sup_x|\bar F_{\widehat \xi}(x) - \Phi(x)|\overset{P}{\rightarrow} 0$. 
 Since for any finite dimensional vector $(\xi_{t_1}^*,\cdots,\xi_{t_d}^*)$, each of the elements are i.i.d. sampled from $\bar F_{\widehat\xi}$, then $(\xi_{t_1}^*,\cdots,\xi_{t_d}^*)$ converge in distribution(in bootstrap world) to a d-vector with i.i.d. standard normals in probability. By Theorem 3.29 of \cite{Kallenberg1997}, assertion in (\ref{eq:418}) holds.The remaining proof goes back to the setting in (1).
\end{proof}
\quad
\begin{proof}[Proof of Theorem \ref{th:46}]
	The proof is based on several key components that can be checked easily:
	\begin{enumerate}
		\item  Finite dimensional convergence in distribution shown in Lemma \ref{lm:45}.
		\item The bootstrap sample $Y_t^*$ are $\lfloor c_\kappa l\rfloor $-dependent with respect to $P^*$ due to $i.i.d.$ sampling of the $\xi_t^*$s and bandedness of $\wwidehat{\Sigma}_n$. 
		Also $\exists \delta>0, \forall t, E^*(Y_t^*)^{2 + \delta} = O_p(1)$ from $L^p -m$ approximable assumption of $Y_t$(with $p>2$) which implies uniform integrability of $(Y_t^*)^2$ (with respect to $n$). Also (C2) holds for $Y_t^*$ in probability.
		\item  Central limit theorem holds by Lemma \ref{th:41} for $\bar Y_n^*$ in probability.
	\end{enumerate}
	The consequences of 1 \& 2 above are two fold. One, the long run variance of $Y_t^*$ can be 
	written as  
	$(\sigma_\infty^*)^2 = \sum_{k = -\lfloor c_\kappa l\rfloor}^{\lfloor c_\kappa l\rfloor} \gamma_Y^*(k)$ which is absolutely convergent in probability as $n\rightarrow\infty$ by Lemma \ref{th:41}. Second, by similar arguments to the proof of Theorem 3.3 in \cite{buhlmann1997},
	we have  $\forall M>0$,  $\sum_{k= -M}^M |\gamma_Y^*(k) -  \gamma_Y(k)| = o_p(1)$. Note that $\sum_{k\in \mathbb{Z}}\gamma_Y(k)$ is also absolutely convergent. All together, they imply $\sum_{k = -\lfloor c_\kappa l\rfloor}^{\lfloor c_\kappa l\rfloor} |\gamma_Y^*(k) - \gamma_Y(k)| = o_p(1)$, and as a consequence, $(\sigma_\infty^*)^2 \overset{P}{\rightarrow}\sigma_\infty^2$.  
	
	The remaining arguments are standard.
	
\end{proof}

\quad

\begin{proof}[Proof of Theorem \ref{th:49}]
	The kernel smoothed spectral density estimator
	\begin{equation}
				\widehat f_{sp.d}(\omega) = \sum_{j\leq  n} \tilde\kappa_h(\omega - \omega_j) I_n(\omega_j)
	\end{equation}
	can be rewritten as the lag-window spectral density estimator with asymptotically negligible error (even after $\sqrt{nh}$- scaling), i.e.:
	\begin{equation}
		\widehat f_{sp.d}(\omega) = \frac{1}{2\pi}\sum_{k\in\mathbb{Z}} K_h(k)\widehat{ \gamma}(k)e^{-i\omega k} + O\left(\frac{1}{n}\right).
	\end{equation}
	with $\widehat{ \gamma}(k) = \frac{1}{n}\sum_{t=1}^n (Y_t - \bar Y) (Y_{t+k} - \bar Y)$ and $K_h(k) = \sum_{j=1}^n \tilde\kappa_h(\omega - \omega_j)e^{i(\omega_j-\omega)k}$.
	See \cite{politis1992}.
	Under assumption (A12) and Lemma \ref{lm:33}, also with previous results holding, $\{Y_t\}$ and $\{Y_t^*\}$ satisfy (C0) with $p=4$; with additional assumptions (A13) and (A14) holding, conditions  in Theorem 2, \cite{liu_wu_2010} are satisfied. Therefore asymptotic normality in the sense of equation (\ref{eq:4226}) holds for both $\widehat{f}_{sp.d}(\omega)$ and $\widehat{f}_{sp.d}^*(\omega)$ with limiting variance $\sigma_\omega^2 = (\eta_\omega + 1) f^2(\omega)\int \tilde\kappa^2(u) du$ and $(\sigma_\omega^{*})^2 = (\eta_\omega + 1) (f^*(\omega))^2\int \tilde\kappa^2(u) du$ respectively. Since $|f_{sp.d}(\omega) - f_{sp.d}^*(\omega)|  \leq \frac{1}{2\pi} \sum_{k\in\mathbb{Z}} |(\gamma_Y(k) - \gamma_Y^*(k))e^{i\omega k}|\leq \sum_{k\in\mathbb{Z}} |\gamma_Y(k) - \gamma_Y^*(k)| \overset{P}{\rightarrow} 0$ and both $f_{sp.d}(\omega)$ and $f_{sp.d}^*(\omega)$ are bounded and bounded a.s. , $|f_{sp.d}^2(\omega) - (f_{sp.d}^*(\omega))^2| \overset{P}{\rightarrow} 0$, which implies $(\sigma_\omega^*)^2\overset{P}{\rightarrow}\sigma_\omega^2$.
\end{proof}

\quad

\begin{proof}[Proof of Theorem \ref{th:51}]
We prove it for $\widehat F = \bar F$. The case for the kernel estimator $\widehat F_h$ is similar with relative assumptions assumed. We mainly show that the distribution of bootstrap predictive root converges in probability to the true distribution in a conditional sense. i.e., $Y_{n+1}^* - \widehat Y_{n+1}^* \overset{d^*}{\rightarrow} Y_{n+1} - \widehat Y_{n+1}$ in probability, conditioning on past values $\underline Y_n$.

 Firstly, conditioning on $\underbar Y_n$, $Y_{n+1}^* \overset{d^*}{\rightarrow} Y_{n+1}$ in probability:\\
\begin{equation}
\begin{split}
Y_{n+1}^* = \Bar F_Y^{-1}(\Phi(Z_{n+1}^*)) & = \left(\Bar F_Y^{-1}(\Phi(Z_{n+1}^*)) -  F_Y^{-1}(\Phi(Z_{n+1}^*))\right) 
+ F_Y^{-1}(\Phi(Z_{n+1}^{*})) 
\end{split}
\end{equation}
We show that first term converges to 0 in probability and the distribution of second term converges to that of $Y_{n+1}$, in probability.
Let $U = \Phi(Z_{n+1}^*)$. Then $\Bar{F}_{Y}^{-1}(U) = Y_{(i)}$, where $\frac{i-1}{n}<U\leq\frac{i}{n}$. It is equivalent to show $Y_{(i)} - F_Y^{-1}(U)\underset{P}{\rightarrow} 0$. Consider $F_Y(Y_{(i)}) - U = \left(F_Y(Y_{(i)}) - \Bar F_Y(Y_{(i)})\right) +\left( \Bar F_Y(Y_{(i)})- U\right)$. The first term is $O_p(\frac{1}{\sqrt{n}})$ by Lemma \ref{th:41}. For the second term, 
$\Bar F_Y(Y_{(i)})- U = \frac{i}{n} - U$, where $U\in(\frac{i-1}{n},\frac{i}{n}]$. Hence the second term goes to 0. Therefore $F_Y(Y_{(i)}) - U \rightarrow 0$ in probability. Since $F^{-1}_Y$ is continuous, $Y_{(i)} - F_Y^{-1}(U)\overset{P}{\rightarrow} 0$.

To show $F_Y^{-1}(\Phi(Z_{n+1}^*)) \overset{d^*}{\rightarrow} Y_{n+1} = F_Y^{-1}(\Phi(Z_{n+1}))$,  we need to show conditioning on $\underbar Y_n$, $Z_{n+1}^*\overset{d^*}{\rightarrow} Z_{n+1}$ in probability. For MF bootstrap, let $n(\cdot,\cdot): (\underline Y_n,\xi_{n+1}^*)\rightarrow Z_{n+1}^*$ with its theoretical analogue $L(\cdot,\cdot):  (\underline Y_n,\xi_{n+1})\rightarrow Z_{n+1}$. The formula of $L_n(\cdot,\cdot)$ is estimated from $\underline Y_n$ and therefore depends on the past values $\underline Y_n$; While $L(\cdot,\cdot)$ is the theoretical data-generating mechanism that links the past values $\underline Y_n$ and innovation $\xi_{n+1}$ so it does not depend on $\underline Y_n$.
By results in Lemma \ref{lm:45} and assumption (A9), $L_n(x,y)$ converges to $L(x,y)$ in probability; Also the distribution of $\xi_{n+1}^*$ converges to $\Phi$ for MF bootstrap. As a result of continuous mapping theorem, $L_n(\underline Y_n,\xi_{n+1}^*)\overset{d^*}{\rightarrow} L(\underline Y_n,\xi_{n+1})$ in probability.

For LMF bootstrap, we only need to show the normal parameters in step 2 of \textbf{Algorithm 3} converge to their theoretical analogue, in probability. As an example, we show $\wwidehat\Sigma_{21}\wwidehat\Sigma_{11}^{-1}\tilde{\widehat{\underbar Z}}_n \rightarrow \Sigma_{21}\Sigma_{11}^{-1}\underbar Z_n$ in probability. Let $c_n = \Sigma_{21}\Sigma_{11}^{-1}$\& $\wwidehat c_n = \wwidehat\Sigma_{21}\wwidehat\Sigma_{11}^{-1}$. we have
\begin{equation}
\begin{split}
|\wwidehat{c}_n\tildehat{\underbar Z}_n - c_n\underbar Z_n| &\leq \sum_{i=1}^n |\wwidehat c_{n,i}\tildehat Z_i - c_{n,i}Z_i|\\
&\leq \sum_{i=1}^n |\wwidehat c_{n,i}||\tildehat Z_i - Z_i| + |\sum_{i=1}^n \left(c_{n,i} - \wwidehat c_{n,i}\right)Z_i|\\
\end{split}
\end{equation}
The first term is bounded by $ \left(\sum_{i=1}^n |\wwidehat c_{n,i}|\right)\sup_{i}|\tildehat Z_i - Z_i|$, which is $o_p(1)$ by continuous mapping theorem and $\sum_{i=1}^n |\wwidehat c_{n,i}|$ is bounded in probability for large $n$. The second term has $\mathcal{N}(0,(c_n - \wwidehat c_n)\Sigma_n(c_n - \wwidehat c_n)^T)$ distribution, where the variance is bounded by $\lambda_{\Sigma_n}^{max}\norm{c_n - \wwidehat c_n}_2^2$. Thus if $\norm{c_n - \wwidehat c_n}_2 \underset{P}{\rightarrow}  0$, then second term converge to 0 in probability.
\begin{equation}
\norm{c_n - \wwidehat c_n}_2\leq \norm{\wwidehat \Sigma_n^{-1} - \Sigma_n^{-1} }_{op} \norm{\wwidehat \Sigma_{21}}_2 + \norm{\wwidehat \Sigma_{21} - \Sigma_{21}}_2\norm{\Sigma_{n}^{-1}}_{op}
\end{equation}   
Since
$\norm{\wwidehat\Sigma_n^{-1} - \Sigma_n^{-1}}_{op}\underset{P}{\rightarrow} 0$,
by previous results, $\norm{\wwidehat \Sigma_{21} - \Sigma_{21}}_2\underset{P}{\rightarrow}0$. Then $\norm{c_n - \wwidehat c_n}_2\underset{P}{\rightarrow} 0$.
The proof for consistency of the second parameter follows a  similar fashion. 

What remains to be shown is conditioning on $\underline Y_n$,
\begin{equation}
		\widehat Y_{n+1}^* \overset{P^*}{\rightarrow} \widehat Y_{n+1}.
\end{equation}
(In probability),
where $\widehat Y_{n+1}$ is the self-chosen predictor. In $L^2$ optimal sense, $\widehat Y_{n+1} = E(Y_{n+1}|\underline Y_n) = E(F_{Y}^{-1}(\Phi(Z_{n+1}))|\underline Y_n)$. Therefore it is sufficient to show that:
\begin{equation}\label{eq:728}
E^*((\Bar{F}^*)^{-1}(\Phi(Z_{n+1}))|\underline Y_n) \overset{P^*}{\rightarrow} E(F_{Y}^{-1}(\Phi(Z_{n+1}))|\underline Y_n)
\end{equation}
Where the expectation is taken with respect to $(Z_{n+1}|\underline Y_n)^*$ and $(Z_{n+1}|\underline Y_{n})$  as defined in \textbf{Algorithm} 2 and 3.
Since directly calculating the expectation in our setup is impossible, we use the same approach as previously used in the proof of Theorem \ref{th:46} by showing:
\begin{enumerate}
	\item $(Z_{n+1}|\underline Y_n)^* \overset{d^*}{\rightarrow} (Z_{n+1}|\underline Y_{n})$.
	\item Consistency of $(\bar F^*)^{-1}$ to $F_Y^{-1}$.
	\item Uniform integrability of $(\Bar{F}^*)^{-1}(\Phi(Z_{n+1})$ and $F_{Y}^{-1}(\Phi(Z_{n+1}))$ conditioning on $\underline Y_n$.
\end{enumerate}
The second
First of all, we need to show that $(Z_{n+1}|\underline Y_n)^*\overset{d^*}{\rightarrow} Z_{n+1}|\underline Y_n$ in probability. In LMF bootstrap setup as in \textbf{Algorithm} 3, this can be checked by showing convergence of the bootstraped autocovariance matrix $\wwidehat{\Sigma}_n^*$ to $\Sigma_n$ in probability. Since $\{Y_t^*\}_{t=1}^n$ are m-dependent and have finite $p^{th}$ moment in probability, Theorem \ref{th:37} applies to $\{Y_t^*\}_{t=1}^n$ and thus
\begin{equation}
		\norm{\wwidehat{\Sigma}_n^* - \Sigma_n^*}_{op}\overset{P^*}{\rightarrow} 0
\end{equation}
where $ \Sigma_n^*$ is the autocovariance matrix of the bootstrap samples $\{Y_t^*\}$. Also by Lemma \ref{lm:45} and uniform integrability of $(Y_t^*)^2$, We have 
$\sum_{k\in\mathbb{Z}} |\sigma^*(k) - \sigma(k)|\overset{P^*}{\rightarrow} 0$. This implies 
	$\norm{{\Sigma}_n^* - \Sigma_n}_{op}\overset{P}{\rightarrow} 0$. Therefore by triangle inequality, 
	\begin{equation}\label{eq:723}
		\norm{\wwidehat{\Sigma}_n^* - \Sigma_n}_{op}\overset{P^*}{\rightarrow} 0.
	\end{equation}

In MF bootstrap setup as in \textbf{Algorithm} 2, we need to show $\forall z\in\mathbb{R}$,
$(\widehat{F}_Z^{(n+1)})^*(z)\overset{P^*}{\rightarrow} \widehat{F}_Z^{(n+1)}(z)$. This can also be shown through equation (\ref{eq:417}) in Lemma \ref{lm:45} and a slightly stronger result than equation (\ref{eq:723}) above, i.e.: we assume $(\log n)^2\norm{\wwidehat{\Sigma}_n^* - \Sigma_n}_{op}\overset{P^*}{\rightarrow} 0.$

Since $\bar F_Y^*$ is the empirical CDF of $\{Y_t^*\}$ for which the CDF $F^*_Y$ satisfies $\sup_{y\in\mathbb{R}}|F^*_Y(y) - F_Y(y)|\overset{P}{\rightarrow} 0$ by Lemma \ref{lm:45}, also along with $m-$dependence of $\{Y_t^*\}$,  
$\sup_{y\in\mathbb{R}} |\bar F_Y^*(y) - F^*_Y(y)|\rightarrow 0$
 almost surely with respect to $P^*$. Combining these two we have $\sup_{y\in\mathbb{R}}|\bar F_Y^*(y) - F_Y(y)|\rightarrow 0$ in probability. By Lemma \ref{lm:71} and continous mapping theorem used in Lemma \ref{lm:45}, $(\Bar{F}_{Y}^*)^{-1}(\Phi(Z_{n+1}))$ with 
$Z_{n+1}$ having the distribution of $(Z_{n+1}|\underline Y_n)^*$ converges in distribution to  $F_{Y}^{-1}(\Phi(Z_{n+1}))$ with $Z_{n+1}$ following the conditional distribution of $Z_{n+1}|\underline Y_n$, in probability. Also by finite $p$th moment of $Y_t$ we have uniform integrability of $(\Bar{F}_{Y}^*)^{-1}(\Phi(Z_{n+1}))$(in probability) and $F_{Y}^{-1}(\Phi(Z_{n+1}))$. Thus equation (\ref{eq:728}) holds. 
For the $L^1$ optimal predictor mentioned in Section 1, proving $\widehat Y_{n+1}^*\overset{P^*}{\rightarrow} \widehat Y_{n+1}$ follows in a similar fashion but may require additional assumptions; for example, continuity and strict monotonicity of the conditional distribution $Y_{n+1}|\underline Y_n$ is necessary to ensure consistency of $\widehat Y_{n+1}$, which is the conditional sample median. 

Summing up the previous results and by Slutsky's theorem, the bootstrap predictive root converges in distribution to the true predictive root(in probability). Since the distribution of the true predictive root is continuous,  we have consistency of the bootstrap quantiles.
\end{proof}

\end{document}